\documentclass[a4paper,12pt]{amsart}

\usepackage{amscd}
\usepackage[all]{xy}

\usepackage{mathrsfs}
\usepackage{stmaryrd}
\usepackage{amssymb, amsfonts, latexsym, mathtools}
\usepackage{amsmath, relsize}
\usepackage{amsthm}

\DeclareMathOperator{\Spec}{Spec}

\DeclareMathOperator{\Gal}{Gal}
\DeclareMathOperator{\id}{id}

\DeclareMathOperator{\End}{End}
\DeclareMathOperator{\Hom}{Hom}
\DeclareMathOperator{\Ext}{Ext}
\DeclareMathOperator{\Rep}{Rep}

\DeclareMathOperator{\Ker}{Ker}

\DeclareMathOperator{\Aut}{Aut} 
 
\DeclareMathOperator{\Ind}{Ind}

\DeclareMathOperator{\Irr}{Irr}

\DeclareMathOperator{\Image}{Im}
\DeclareMathOperator{\Spa}{Spa}
\DeclareMathOperator{\Lie}{Lie}

\DeclareMathOperator{\Sym}{Sym}

\DeclareMathOperator{\Spd}{Spd}
\DeclareMathOperator{\Perf}{Perf}

\DeclareMathOperator{\Gr}{Gr}
\DeclareMathOperator{\GL}{GL}
\DeclareMathOperator{\SL}{SL}
\DeclareMathOperator{\GSp}{GSp}
\DeclareMathOperator{\GU}{GU}
\DeclareMathOperator{\Bun}{Bun}
\DeclareMathOperator{\lis}{lis}

\DeclareMathOperator{\Ad}{Ad}
\DeclareMathOperator{\Cont}{Cont}
\DeclareMathOperator{\cInd}{c-Ind}
\DeclareMathOperator{\Div}{Div}
\DeclareMathOperator{\Hck}{Hck}

\DeclareMathOperator{\trg}{trg}

\newcommand{\sHom}{\mathscr{H}\!\mathit{om}}
\newcommand{\cHck}{\mathcal{H}\hspace{-0.1em}\mathit{ck}}

\newcommand{\solid}{\mathsmaller{\blacksquare}}

\newcommand{\setmid}{\mathrel{}\middle|\mathrel{}}

\newcommand{\bD}{\mathbb{D}}

\newcommand{\bF}{\mathbb{F}}
\newcommand{\bG}{\mathbb{G}}

\newcommand{\bL}{\mathbb{L}}

\newcommand{\bQ}{\mathbb{Q}}

\newcommand{\bX}{\mathbb{X}}
\newcommand{\bZ}{\mathbb{Z}}

\newcommand{\bfC}{\mathbf{C}}

\newcommand{\cB}{\mathcal{B}}
\newcommand{\cC}{\mathcal{C}}
\newcommand{\cD}{\mathcal{D}}
\newcommand{\cE}{\mathcal{E}}
\newcommand{\cF}{\mathcal{F}}
\newcommand{\cG}{\mathcal{G}}
\newcommand{\cH}{\mathcal{H}}
\newcommand{\cI}{\mathcal{I}}

\newcommand{\cN}{\mathcal{N}}
\newcommand{\cO}{\mathcal{O}}
\newcommand{\cP}{\mathcal{P}}

\newcommand{\cS}{\mathcal{S}}
\newcommand{\cT}{\mathcal{T}}

\newcommand{\cW}{\mathcal{W}}

\newcommand{\cZ}{\mathcal{Z}}

\newcommand{\fm}{\mathfrak{m}}

\newcommand{\fw}{\mathfrak{w}}

\newcommand{\fA}{\mathfrak{A}}

\newcommand{\sC}{\mathscr{C}}

\newcommand{\rmab}{\mathrm{ab}}
\newcommand{\rmur}{\mathrm{ur}}

\newcommand{\ol}{\overline}
\newcommand{\ul}{\underline}
\newcommand{\wh}{\widehat}
\newcommand{\wt}{\widetilde}
\newcommand{\lra}{\longrightarrow}

\newcommand{\cf}{\textit{cf.\ }}

\newtheorem{thm}{Theorem}[section]
\newtheorem{dfn}[thm]{Definition}
\newtheorem{prop}[thm]{Proposition}
\newtheorem{lem}[thm]{Lemma}
\newtheorem{conj}[thm]{Conjecture}
\newtheorem{ex}[thm]{Example}
\newtheorem{rem}[thm]{Remark}

\makeatletter

\@addtoreset{equation}{section}
\makeatother

\begin{document}

\title{On the geometrization of the local Langlands correspondence}
\author{Naoki Imai}
\date{} 
\maketitle

\begin{abstract}
This is a survey paper on the geometrization of the local Langlands correspondence by Fargues--Scholze. 
\end{abstract}

\section{Introduction}\label{sec1}
The purpose of this paper is to explain about Fargues--Scholze's monumental paper \cite{FaScGeomLLC} on the geometrization of the local Langlands correspondence, which implements and advances the ideas Fargues gave in \cite{FarGover}. From the viewpoint of the local Langlands correspondence, the main results are the construction of semisimple L-parameters and the formulation of the geometric categorical local Langlands conjecture. 
The former and the latter realize Lafforgue's construction \cite{LafChtred} in the function field case, and a formulation of Arinkin--Gaitsgory's categorical conjecture \cite{ArGaSinggLc} in the geometric Langlands correspondence, respectively, for non-archimedesque local fields.
The most basic idea is to replace the algebraic curves in the geometric Langlands correspondence by Fargues--Fontaine curves, which are arithmetic curves introduced in \cite{FaFoCfv}.
However, unlike algebraic curves, relative Fargues--Fontaine curves do not have structure morphisms, so it is necessary to consider not only the Fargues--Fontaine curve itself but also the moduli space of Cartier divisors on it (\cf Remark \ref{rem:algvsFF}).
When $G$ is a connected reductive algebraic group over a non-Archimedesian local field, the moduli spaces of $G$-bundles on Fargues--Fontaine curves and the moduli spaces of the Cartier divisors above are formulated as spaces in positive characteristics, and we need various concepts of spaces introduced in \cite{SchEtdia} and \cite{ScWeBLp} to deal with them. 
The reason why we may consider the spaces in positive characteristics when we want to deal with the problem of a $p$-adic field in characteristic $0$ is fundamentally due to the tilting equivalence introduced in \cite{SchPerf}, which relates the characteristics $0$ and $p$. 

There are several formulations of the local Langlands correspondence, but the formulation (\cite{BMOBGpar}, \cite{KalSiso}) using the Kottwitz set $B(G)$ of $G$ is most compatible with the geometrization of the local Langlands correspondence explained in this paper. 
This is related to the fact that isomorphisms of $G$-bundles on the Fargues--Fontaine curve can be classified using $B(G)$.
In the formulation using $B(G)$, we take a quasi-split $G$ and consider the inner forms of $G$ determined from the elements of $B(G)$. However, not all connected reductive algebraic groups appear as such inner forms, and a formulation with rigid inner forms \cite{KalRiginn} is necessary to deal with general connected reductive algebraic groups\footnote{However, it is known that the two formulations are equivalent in some sense (\cite{KalRigvsiso}).}. A generalization of the geometrization of the local Langlands correspondence to the rigid inner forms is given in \cite{FargerKal}, but we will not discuss it further in this paper. 

Perhaps the best overview of the geometrization of the local Langlands correspondence is the introduction of \cite{FaScGeomLLC}. The reader who is interested should read it as well, since it is very well-written including historical information. Fargues--Scholze themselves also wrote an overview paper \cite{FaScLangIHES}, in which they explain a Jacobian criterion, one of the key technical points. 
In this paper, we have tried to explain contents which are not covered in these overviews. 
In \cite{HanBeicat}, more advanced topics of the categorical conjecture are discussed.

\section{Local Langlands correspondence}

In this section, we will explain the statement of the local Langlands correspondence. 
First, we explain L-groups.
A Borel pair is a pair of a Borel subgroup and a maximal torus contained in the Borel subgroup of a reductive group scheme \footnote {A reductive group scheme is a flat group scheme whose fibers are connected reductive algebraic groups.}. 
For the splitting reductive group scheme $\cG$ and its Borel pair $(\cB,\cT)$ over a connected scheme, 
we denote the attached based root data\footnote{This is what is called donn\`{e}e radicielle in \cite[XXIII, 1.5]{SGA3-3}. Here we follow the terminology of \cite[1.5]{ConRgrsch}. Also, although \cite[XXIII, 1.5]{SGA3-3} defines the based root data for pinned reductive algebraic group schemes, only Borel pairs are actually used in the construction.} 
by $\mathrm{BR}(\cG,\cB,\cT)$ (\cite[XXIII, 1.5]{SGA3-3}).
For a split torus $\cT$ over a field, we denote by $X^*(\cT)$ and $X_*(\cT)$ the character group and the cocharacter group of $\cT$.
For a Borel pair $(\cB,\cT)$, $(\cB',\cT')$ of a split connected reductive algebra group $\cG$ over a field $k$, let $g \in \cG(k)$ such that $\Ad (g)(\cB,\cT)=(\cB',\cT')$, then
the isomorphism $\mathrm{BR}(\cG,\cB,\cT)\cong \mathrm{BR}(\cG,\cB',\cT')$ induced by $\Ad (g)$ is independent of the choice of $g$. 
Using these as transition morphisms, we put 
\[
 \mathrm{BR}(\cG)=\mathrm{colim}_{(\cB,\cT)} \mathrm{BR}(\cG,\cB,\cT) . 
\] 
We write 
$\mathrm{BR}(\cG)=(\bX_{\cG},\Phi_{\cG},\Delta_{\cG},\check{\bX}_{\cG},\check{\Phi}_{\cG},\check{\Delta}_{\cG})$, and define its dual by 
\[
\mathrm{BR}(\cG)^{\vee}=(\check{\bX}_{\cG},\check{\Phi}_{\cG},\check{\Delta}_{\cG},\bX_{\cG},\Phi_{\cG},\Delta_{\cG}) 
\]
(\cf \cite[\S 1.5]{ConRgrsch}). 

Let $F$ be a non-Archimedean local field, $\cO_F$ be a ring of integers in $F$, and $\fm_F$ be the maximal ideal of $\cO_F$.
Let $k_F$ be the residue field of $F$ and let $q$ be the order of $k_F$.
Let $F^{\mathrm{sep}}$ be a separable closure of $F$ and $k_F^{\mathrm{sep}}$ be its residue field. 
Let $\sigma \in \Gal (k_F^{\mathrm{sep}}/k_F)$ be the $q$-th power Frobenius map.
Let $I_F$ be the kernel of the natural surjection $\Gal (F^{\mathrm{sep}}/F)\to \Gal (k_F^{\mathrm{sep}}/k_F)$ and $W_F$ be the inverse image of $\langle \sigma \rangle$ under this sujection. 
We equip $I_F$ with the topology induced by $\Gal (F^{\mathrm{sep}}/F)$, and
$W_F$ with the topology such that $I_F \subset W_F$ is an open subgroup.
We call $I_F$ the inertia group of $F$ and $W_F$ the Weil group of $F$. 
For $w \in W_F$, we define $d_F (w)\in \mathbb{Z}$ so that the image of $w$ in $\Gal (k_F^{\mathrm{sep}}/k_F)$ is $\sigma^{-d_F (w)}$. 

Let $G$ be a connected reductive algebraic group on $F$.
Then $W_F$ acts naturally on $\mathrm{BR}(G_{F^{\mathrm{sep}}})$.
We take a pinned reductive algebraic group scheme $(\wh{G},\wh{B},\wh{T},\{\wh{X}_{\alpha} \}_{\alpha \in \Delta_{\wh{G}}})$ and an isomorphism $\mathrm{BR}(\wh{G},\wh{B},\wh{T})\cong \mathrm{BR}(G_{F^{\mathrm{sep}}})^{\vee}$ using the equivalence of categories \cite[XXV, Th\'{e}or\`{e}me 1.1]{SGA3-3}. 
These are unique up to uniquely determined isomorphisms. 
By 
\begin{align*}
 W_F \to \Aut (\mathrm{BR}(G_{F^{\mathrm{sep}}})) \cong \Aut (\mathrm{BR}(G_{F^{\mathrm{sep}}})^{\vee}) &\cong \Aut (\mathrm{BR}(\wh{G},\wh{B},\wh{T})) \\ 
 &\cong \Aut (\wh{G},\wh{B},\wh{T},\{\wh{X}_{\alpha} \}_{\alpha \in \Delta_{\wh{G}}}) \hookrightarrow \Aut (\wh{G}) 
\end{align*}
we define the action of $W_F$ on $\wh{G}$. 
We call $\wh{G}$ the dual group of $G$ over $\bZ$. 
Let $\cW_F$ be the Weil group scheme (\cite[(4.1)]{TatNtb}) over $\bZ$. Then $\cW_F$ acts on $\wh{G}$. 
We put ${}^L G=\wh{G} \rtimes \cW_F$ and call it the L-group of $G$. 

Next, we explain the L-parameters following \cite[\S 8]{BorAutL}. 
Let $\bfC$ be an algebraically closed field of characteristic $0$.

\begin{dfn}
Let $\cP$ be a subgroup scheme of ${}^L G_{\bfC}$. 
If $\cP \hookrightarrow {}^L G_{\bfC} \to \cW_{F,\bfC}$ is a surjection and $\cP \cap \wh{G}_{\bfC}$ is a parabolic subgroup of $\wh{G}_{\bfC}$, then we say that $\cP$ is a parabolic subgroup of ${}^L G_{\bfC}$. 
\end{dfn}

Let $\cP$ be a parabolic subgroup of ${}^L G_{\bfC}$.
For $w \in W_F$, taking $(g,w) \in \cP (\bfC)$, we have 
$w(\cP \cap \wh{G}_{\bfC})=\Ad(g^{-1})(\cP \cap \wh{G}_{\bfC})$ because $\Ad((g,w))(\cP \cap \wh{G}_{\bfC})=\cP \cap \wh{G}_{\bfC}$. 
Hence the $\wh{G}(\bfC)$-conjugacy class of $\cP \cap \wh{G}_{\bfC}$ is $W_F$-stable.
Using the correspondence between the conjugacy classes of parabolic subgroups and sets of simple roots (\cite[5.14]{BoTiGred}), the  $\wh{G}(\bfC)$-conjugacy class of $\cP \cap \wh{G}_{\bfC}$ and $\Delta_{\wh{G}_{\bfC}} \cong \check{\Delta}_{G_{F^{\mathrm{sep}}}} \cong \Delta_{G_{F^{\mathrm{sep}}}}$ determine a $W_F$-stable $G_{F^{\mathrm{sep}}}$-conjugacy class of parabolic subgroups of $G_{F^{\mathrm{sep}}}$.
If this $G(F^{\mathrm{sep}})$-conjugacy class contains a parabolic subgroup of $G$ defined over $F$, 
we say that $G(F^{\mathrm{sep}})$-conjugacy class is relevant. 
If $G$ is quasi-split, then any parabolic subgroup of ${}^L G_{\bfC}$ is relevant (\cite[6.3]{BoTiGred}). 

\begin{dfn}
We define a finite Galois extension $F^*$ of $F$ by $W_{F^*}=\Ker (W_F \to \Aut (\wh{G}))$. 
We say that $g \in {}^L G (\bfC)$ is semisimple if its image in $\wh{G}_{\bfC} \rtimes \Gal (F^*/F)$ is semisimple. 
\end{dfn}

\begin{dfn}\label{dfn:SLLp}
An L-parameter of $\SL_2$-type in $\bfC$-coefficients for $G$ is the morphism 
$\phi \colon \SL_{2,\bfC} \times \cW_{F,\bfC} \to {}^L G_{\bfC}$ of the group scheme over $\bfC$ that satisfies the following: 
\begin{enumerate}
	\item 
	$\phi$ is compatible with the projection on $\cW_{F,\bfC}$.
	\item \label{en:Frobss}
	For any element $w$ of $\cW_F(\bfC)$, $\phi (1,w)$ is semisimple.
	\item\label{en:relev}
	If $\phi$ factors through a parabolic subgroup $\cP$ of ${}^L G_{\bfC}$, then $\cP$ is relevant.
\end{enumerate}
We say that two L-parameters are equivalence if they are conjugate under $\wh{G}(\bfC)$. 
\end{dfn}

The condition of Definition \ref{dfn:SLLp} \ref{en:Frobss} is equivalent to the condition that the image of a lift of the Frobenius element is semisimple, which is called Frobenius semisimplicity. 

\begin{dfn}
Let $\phi$ be an L-parameter and $S_{\phi}=Z_{\wh{G}_{\bfC}}(\phi )$. 
\begin{enumerate}
\item
We say that $\phi$ is discrete if $S_{\phi}/Z(\wh{G}_{\bfC})^{W_F}$ is finite.
\item 
We say that $\phi$ is cuspidal if $\phi$ is discrete and the restriction of $\phi$ to $\SL_{2,\bfC}$ is trivial.
\end{enumerate}
\end{dfn}

Let $\Phi (G)$ be the set of isomorphisms classes of L-parameters of $G$. For a locally profinite group $\mathsf{G}$, let $\Pi(\mathsf{G})$ denote the isomorphism classes of smooth\footnote{The action of $\mathsf{G}$ on an additive group $M$ is smooth if for any $m \in M$, the stabilizer of $m$ is an open subgroup of $G$.} irreducible representations over $\bfC$ of $\mathsf{G}$. 
First, we state the claim of the local Langlands correspondence in a crude form.

\begin{conj}\label{conj:LL}
We fix $c \in \bfC$ such that $c^2=q$. 
Then, there exists a finite-to-one natural surjection  
\[
 \mathrm{LL}_G \colon \Pi (G(F)) \to \Phi (G) . 
\] 
\end{conj}

For the various expected properties that $\mathrm{LL}_G$ of Conjecture \ref{conj:LL} should satisfy, see \cite[\S 10]{BorAutL}.
For $\pi \in \Pi (G(F))$, we call $\mathrm{LL}_G(\pi)$ the L-parameter of $\pi$. Furthermore, the pullback of  $\mathrm{LL}_G(\pi)$ under 
\[
 \SL_{2,\bfC} \times \cW_{F,\bfC} \to \SL_{2,\bfC} \times \cW_{F,\bfC} ;\ (g,w) \mapsto \left( \begin{pmatrix}
 c^{-d_F (w)} & 0 \\ 0 & c^{d_F (w)}
 \end{pmatrix},w \right) 
\]
is called the semisimple L-parameter of $\pi$\footnote{We define it this way because we want to consider an L-parameter obtained from the corresponding Weil--Deligne L parameters by forgetting the monodromy operator.} 

Next, we explain a refinement of the local Langlands correspondence.
Let $\breve{F}$ denote the completion of the maximal unramified extension $F^{\mathrm{ur}}$ of $F$.
The $q$-th power Frobenius map $\sigma$ can be regarded as an element of $\Gal (F^{\mathrm{ur}}/F)$ by the isomorphism $\Gal (F^{\mathrm{ur}}/F)\cong \Gal (F^{\mathrm{sep}}/k_F)$, and act on $\breve{F}$ naturally. 
For $b,b' \in G(\breve{F})$, we say that $b$ and $b'$ are $\sigma$-conjugate if there exists some $g \in G(\breve{F})$ such that $b'=g b \sigma (g)^{-1}$. 
The set of $\sigma$-conjugacy classes in $G(\breve{F})$ is denoted by $B(G)$, and is called the Kottwitz set of $G$. 
There exists a map 
\[
\kappa_G \colon B(G) \to X^* (Z(\wh{G}_{\bfC})^{W_F})
\]
constructed in \cite[Lemma 6.1]{KotShlam}, which is called the Kottwitz map. 
For $b \in G(\breve{F})$, we define an algebraic group $G_b$ over $F$ by 
\[
G_b (R) =\{ g \in G(R \otimes_F \breve{F}) \mid \Ad(b)(\sigma (g))=g \}. 
\]
If $G_b$ is an inner form of $G$, we say that $b \in G(\breve{F})$ is a basic element. 
The set of $\sigma$-conjugacy classes of the basic elements of $G(\breve{F})$ is denoted by $B(G)_{\mathrm{basic}}$. Then 
\[
\kappa_G|_{B(G)_{\mathrm{basic}}} \colon B(G)_{\mathrm{basic}} \to X^* (Z(\wh{G}_{\bfC})^{W_F}) 
\]
is a bijection (\cite[5.6. Proposition]{KotIso}).

If $b,b' \in G(\breve{F})$ is $\sigma$-conjugate, 
we have an isomorphism $\Ad (g)\colon G_b \cong G_{b'}$ 
taking $g \in G(\breve{F})$ such that $b'=g b \sigma (g)^{-1}$, and the bijection $\Pi (G_b) \cong \Pi (G_{b'})$ determined by this isomorphism is independent of the choice of $g$. 
Using this bijection, we put $\Pi_{[b]} =\varprojlim_{b' \in [b]} \Pi (G_{b'})$.

In the rest of this section, we assume that $G$ is quasi-split. 

\begin{dfn}
Let $\Lambda$ be a ring.
Let $B$ be a Borel subgroup of $G$ and $R_{\mathrm{u}}(B)$ be the unipotent radical of $B$. 
A smooth character $\psi \colon R_{\mathrm{u}}(B)(F)\to \Lambda^{\times}$ is called generic if it satisfies 
\[
\{ g \in B(F) \mid \Ad (g) \psi = \psi \}=Z(G)(F) R_{\mathrm{u}}(B)(F) . 
\]
The pair of Borel subgroup $B$ and generic character $\psi$ is called a Whittaker datum in $\Lambda$-coefficients for $G$. 
\end{dfn}

\begin{dfn}
Let $\fw$ be the $G(F)$-conjugacy class of Whittaker data in $\bfC$-coefficients for $G$. 
Let $\pi$ be a smooth irreducible representation of $G(F)$ over $\bfC$.
We say that $\pi$ is $\fw$-generic if there exist $(B,\psi)\in \fw$ and a non-zero morphism $\pi|_{R_{\mathrm{u}}(B)(F)} \to \psi$ of  $R_{\mathrm{u}}(B)(F)$-representations. 
\end{dfn}

The next conjecture is a refinement of the local Langlands correspondence following \cite{BMOBGpar}, \cite[Conjecture 2.4.1]{KalSiso}.

\begin{conj}\label{conj:LLb}
We fix $c \in \bfC$ such that $c^2=q$ and a $G(F)$-conjugacy class $\fw$ of Whittaker data in $\bfC$-coefficients for $G$.  
Then, there exist a finite-to-one natural map 
\[
 \mathrm{LL}_{[b]} \colon \Pi_{[b]} \to \Phi (G) 
\]
for each $[b] \in B(G)$ and a bijection $\iota_{\fw}$ 
making the diagram 
\[
  \xymatrix{
 	\coprod_{[b] \in B(G)} \Pi_{[b],\phi} 
 	\ar[r]^-{\iota_{\fw}} \ar[d] & 
 	\mathrm{Irr} (S_{\phi}) \ar[d] \\ 
 	B(G) 
 	\ar[r]^-{\kappa_G} & 
 	X^* (Z(\wh{G}_{\bfC})^{W_F}) 
 }
\]
commutative for each L-parameter $\phi$ of $G$, where we put  
$\Pi_{[b],\phi}=\mathrm{LL}_{[b]}^{-1} ([\phi])$, 
$\mathrm{Irr} (S_{\phi})$ denotes the set of isomorphism classes of algebraic irreducible representations of $S_{\phi}$, the left vertical morphism is the natural projection, the right vertical morphism is a map determined by the restriction of the central character.
Under the natural bijection $\Pi_{[1]} \cong \Pi (G)$, the map $\mathrm{LL}_{[1]}$ and $\mathrm{LL}_{G}$ in Conjecture \ref{conj:LLb} are identified. 
Furthermore, if $\phi$ is discrete\footnote{In fact, this should hold for a class of L-parameters broader than discrete.}, then there exists a unique element of $\Pi_{[1],\phi}$ which is $\fw$-generic, and that element corresponds to the trivial representation of $S_{\phi}$ under $\iota_{\fw}$. 
\end{conj}

\section{Satake isomorphism}\label{sec:Satiso}

In this section, we explain the relationship between the Satake isomorphism describing the unramified Hecke algebra and the local Langlands correspondence.

Let $\mathsf{G}$ be a locally profinite group. 
Let $\cC_{\mathrm{c}}^{\infty}(\mathsf{G},\bQ)$ be the $\bQ$-vector space of all compactly supported constant functions on $\mathsf{G}$ that take values in $\bQ$. For $f \in \cC_{\mathrm{c}}^{\infty}(\mathsf{G},\bQ)$ and $g\in \mathsf{G}$, we define $l(g)f, r(g)f \in \cC_{\mathrm{c}}^{\infty}(\mathsf{G},\bQ)$ by $(l(g)f)(x)=f(g^{-1}x)$ and $(r( g)f)(x)=f(xg)$.

\begin{prop}[{\cite[I.2.4]{VigReplmod}}]\label{prop:Haar}
There exists 
$\mu_{\mathsf{G}} \in \Hom_{\bQ}(\cC_{\mathrm{c}}^{\infty}(\mathsf{G},\bQ),\bQ)$ such that $\mu_{\mathsf{G}}(l(g)f)=\mu_{\mathsf{G}}(f)$ for any $f \in \cC_{\mathrm{c}}^{\infty}(\mathsf{G},\bQ)$ and $g\in \mathsf{G}$ and $\mu_{\mathsf{G}}(K)>0$ for any compact open subgroup $K$ of $\mathsf{G}$. 
Moreover, such $\mu_{\mathsf{G}}$ is unique up to positive rational multiples. 
\end{prop}

We call $\mu_{\mathsf{G}}$ in Proposition \ref{prop:Haar} the left Haar measure of $\mathsf{G}$.
For $g \in \mathsf{G}$, there exists uniquely $\delta_{\mathsf{G}}(g)\in \bQ_{>0}$ such that $\mu_{\mathsf{G}}(r(g)f)=\delta_{\mathsf{G}}(g) \mu_{\mathsf{G}}(f)$ for any $f \in \cC_{\mathrm{c}}^{\infty}(\mathsf{G},\bQ)$, since $f \mapsto \mu_{\mathsf{G}}(r(g)f)$ is also a left Haar measure. 
This define $\delta_{\mathsf{G}} \colon \mathsf{G} \to \bQ_{>0}$, which is called the modulus character of $\mathsf{G}$. 

For a compact open subgroup $K$ of $\mathsf{G}$, we have
let $\cH (\mathsf{G},K)$ denote the ring of 
bi-$K$-invariant compactly supported $\bfC$-valued functions on $\mathsf{G}$ with product given by 
\[
(f_1 * f_2)(x)=\int_{\mathsf{G}} f_1(g) f_2(g^{-1}x) \mu_{\mathsf{G}}(g) \quad \quad (x \in \mathsf{G}) . 
\]
This is called the Hecke algebra of $\mathsf{G}$ with respect $K$. 
By \cite[I.3.4]{VigReplmod}, there is an isomorphism 
\begin{equation}\label{eq:EndHiso}
	\End (\cInd_K^{\mathsf{G}} 1)^{\mathrm{op}} \cong \cH (\mathsf{G},K)
\end{equation}
of rings over $\bfC$. 

We say that $G$ is unramified if $G$ extends to a reductive group scheme over $\cO_F$.
In this section, we assume that $G$ is unramified. 

\begin{dfn}
Let $K$ be a subgroup of $G(F)$. 
We say that $K$ is a hyperspecial subgroup of $G(F)$ if $G$ extends to a reductive group scheme $\cG$ over $\cO_F$ and $K$ coincides with $\cG (\cO_F)$. 
\end{dfn}

Let $K$ be a hyperspecial subgroup of $G(F)$, and we take $\cG$ as the above definition. 
By \cite[Corollary 5.2.14]{ConRgrsch} and \cite[XXVI, Corollaire 2.3]{SGA3-3}, we can take a Borel pair $(\cB,\cT)$ of $\cG$. 
We put $B=\cB_F$ and $T=\cT_F$, and let $U$ be the unipotent radical of $B$. 

We fix $c \in \bfC$ such that $c^2=q$ and set $(q^n)^{1/2}=c^n$ for $q^n \in q^{\bZ}$.
Take the left Haar measure $\mu_{U(F)}$ of $U(F)$ such that $\int_{U(F) \cap K} 1 \mu_{U(F)}=1$, and define $\mathsf{S} \colon \cH (G(F),K) \to \cH (T(F),\cT (\cO_F))$ by  
\[
 \mathsf{S} (f)(t)=\delta_{B(F)}(t)^{\frac{1}{2}} \int_{U(F)} f(tu) \mu_{U(F)} . 
\]
Let $A \subset T$ be the maximal split torus. Then there is an isomorphism $\bfC [X_*(A)] \cong \cH (T(F),\cT (\cO_F))$ of $\bfC$-algebras determined by the isomorphism 
\[
X_*(A) \stackrel{\sim}{\longrightarrow} T(F)/\cT (\cO_F) ;\ \mu \mapsto [\mu (\varpi)]
\]
of \cite[9.5]{BorAutL}. 
We put ${}_F W=N_G(A)/T$. Then ${}_F W$ acts naturally on $X_*(A)$.

\begin{thm}[{\cite[Remark 2 of Theorem 3]{SatThesph}, \cite[B.4]{VigCorLGLmodl}}]\label{thm:Satiso}
The map $\mathsf{S}$ induces an isomorphism $\cH (G(F),K) \stackrel{\sim}{\longrightarrow} \bfC [X_*(A)]^{{}_F W}$ of $\bfC$-algebras, 
called the Satake isomorphism. 
\end{thm}

Next, we explain a relationship between the Satake isomorphism and the local Langlands correspondence. 

By \cite[6.7]{BorAutL}, the dual $\wh{T} \to \wh{A}$ of $A \hookrightarrow T$ and $\wh{T} \to \wh{G} \rtimes \sigma;\ t \mapsto (t,\sigma)$ induces an isomorphism 
\begin{equation}\label{eq:AGsig}
 \wh{A}_{\bfC} \sslash {}_F W \cong (\wh{G}_{\bfC} \rtimes \sigma ) \sslash \wh{G}_{\bfC} 
\end{equation}
(\cf \cite[Proposition 4.19]{DHKMModLp}), 
where $\sslash$ denotes the GIT quotient. 
By Theorem \ref{thm:Satiso} and the isomorphism \eqref{eq:AGsig}, we have isomorphisms 
\begin{equation}\label{eq:HGammaiso}
\begin{split}
\cH (G(F),K) &\cong \bfC [X_*(A)]^{{}_F W} \cong \bfC [X^*(\wh{A}_{\bfC})]^{{}_F W} \\
&\cong \Gamma (\wh{A}_{\bfC} \sslash {}_F W,\cO) \cong \Gamma ((\wh{G}_{\bfC} \rtimes \sigma ) \sslash \wh{G}_{\bfC},\cO). 
\end{split} 
\end{equation}
We assume that an L-parameter $\phi$ of $G$ is unramified, namely that  
\[
\cI_{F,\bfC} \times \SL_{2,\bfC} \hookrightarrow \cW_{F,\bfC} \times \SL_{2,\bfC} \stackrel{\phi}{\lra} {}^L G_{\bfC} \to \wh{G}_{\bfC} \rtimes (\cW_F/\cI_F)_{\bfC} 
\]
is trivial. 
Then $(g,\sigma) \in \wh{G}(\bfC) \rtimes \sigma$ is given as the image of $q$-th power Frobenius element under 
$W_F/I_F \to \wh{G}(\bfC) \rtimes \Gal (F^{\rmur}/F)$ induced by $\phi$. 
By this and the isomorphism \eqref{eq:HGammaiso}, we have 
\[
 \theta_{\phi,K} \colon \cH (G(F),K) \cong \Gamma ((\wh{G}_{\bfC} \rtimes \sigma ) \sslash \wh{G}_{\bfC},\cO) \to \bfC . 
\]
Let $\pi_{\phi,K}$ be the smooth irreducible representation of $G(F)$ corresponding under \cite[I.8.9]{VigReplmod} to 
the simple $\cH (G(F),K)$-module given by $\theta_{\phi,K}$. 
We note that the isomorphism class of 
$\pi_{\phi,K}$ depends only on the $G(F)$-conjugacy class of $K$. 

By \cite[Proposition 4.2.6]{ZhuCohLp}\footnote{If $G$ is unramified, being absolutery special parahoric in \cite[Remark 4.2.2]{ZhuCohLp} is equivalent to being hyperspecial. This follows from \cite[1.10.2]{TitRedloc}.}, 
we can attach a $G(F)$-conjugacy class $\mathrm{HS}_{\fw}$ of hyperspecial subgroups of $G(F)$ 
to a $G(F)$-conjugacy class $\fw$ of Whittaker data in $\bfC$-coefficients for $G$. 

For an unramified L-parameter $\phi$ of $G$, 
Conjecture \ref{conj:LL} and Conjecture \ref{conj:LLb} should satisfy the following: 
\begin{itemize}
\item\label{en:unLL} 
$\mathrm{LL}_G^{-1}(\phi)$ is the isomorphism classes of smooth irreducible representations of $G(F)$ such that each of them is isomorphic to $\pi_{\phi,K}$ for some hyperspecial subgroup $K$\footnote{By this, there is a surjection from the set of $G(F)$-conjugacy classes of hyperspecial subgroups of $G(F)$ to  $\mathrm{LL}^{-1}(\phi)$. This map is studied in \cite{MisStrunL}, and is not an injection in general}. 
\item\label{en:unLLb}
For $K \in \mathrm{HS}_{\fw}$, the element $[\pi_{\phi,K}] \in \Pi_{[1],\phi}$ corresponds to the trivial representation of $S_{\phi}$ under $\iota_{\fw}$. 
\end{itemize}

\section{Modulai space of L-parameters}

We explain the moduli space of L-parameters constructed in 
\cite{DHKMModLp}, \cite{FaScGeomLLC}, \cite{ZhuCohLp}. 
Here we follow the formulation in \cite{FaScGeomLLC}. 

\begin{dfn}
Let $\mathrm{Prof}$ be the site of profinite sets with coverings given by finite families of morphisms such that the union of the images of the morphisms in each finite family is the whole set. 
Then the sheaf of sets on $\mathrm{Prof}$ is called a condensed set\footnote{We ignore set-theoretic issues here, but see \cite[Appendix to Lecture II]{SchLecCond} for a precise treatment.}. 
A condensed group, a condensed ring, etc. are defined in the same way. 
\end{dfn}

If $T$ is a topological space, we define the condensed set $T_{\mathrm{c}}$ by associating $C^0 (S,T)$ to a profinite set $S$, where $C^0 (-,-)$ denotes the set of continuous maps. 
If $T$ is a topological group or a topological ring, $T_{\mathrm{c}}$ is defined as a condensed group or a condensed ring similarly. 

Let $\ell$ be a prime number different from $p$. 
In the following, we consider the dual group of $G$ over $\bZ_{\ell}$ and denote it by the same symbol $\wh{G}$.
For a commutative ring $\Lambda$ over $\bZ_{\ell}$, we write $\Lambda_{\mathrm{disc}}$ for $\Lambda$ with discrete topology, and put 
$\Lambda_{\mathrm{c},\ell}=\Lambda_{\mathrm{disc},\mathrm{c}} \otimes_{\bZ_{\ell,\mathrm{disc},\mathrm{c}}} \bZ_{\ell,\mathrm{c}}$.

\begin{dfn}
For a commutative ring $\Lambda$ over $\bZ_{\ell}$, a section $W_{F,\mathrm{c}} \to \wh{G}(\Lambda_{\mathrm{c},\ell}) \rtimes W_{F,\mathrm{c}}$ of the natural projection $\wh{G}(\Lambda_{\mathrm{c},\ell}) \rtimes W_{F,\mathrm{c}} \to W_{F,\mathrm{c}}$ of condensed groups is called an $\ell$-adic L-parameter in $\Lambda$-coefficients. 
\end{dfn}

More concretely, an $\ell$-adic L-parameter in $\Lambda$-coefficients is a section $W_F \to \wh{G}(\Lambda ) \rtimes W_F$ of 
the natural projection $\wh{G}(\Lambda ) \rtimes W_F \to W_F$ such that there are some embedding $\wh{G} \hookrightarrow \GL_n$ and a finitely generated sub-$\bZ_{\ell}$-module $M$ of $\Lambda$ satisfying that 
\[
I_F \hookrightarrow W_F \to \wh{G}(\Lambda ) \rtimes W_F \to \wh{G}(\Lambda ) \hookrightarrow \GL_n (\Lambda ) \hookrightarrow M_n (\Lambda) 
\]
factors through $M_n (M)$ and the factored map is continuous if $M$ is equipped with the $\ell$-adic topology. 

\begin{thm}
The functor sending a commutative ring $\Lambda$ over $\bZ_{\ell}$ to the set of $\ell$-adic L-parameters in $\Lambda$-coefficients is represented by a flat locally complete intersection scheme $Z^1 (W_F,\wh{G})$ over $\bZ_{\ell}$. 
\end{thm}

We explain the construction of $Z^1 (W_F,\wh{G})$. 
First, the following holds. 

\begin{lem}\label{lem:Z1rep}
Suppose that $\Gamma$ is a discrete group and $\Gamma \to \Aut {\wh{G}}$ is given. 
The functor sending a commutative ring $\Lambda$ on $\bZ_{\ell}$ to the set of sections of $\wh{G}(\Lambda) \rtimes \Gamma \to \Gamma$ is represented by an affine scheme $Z^1 (\Gamma,\wh{G})$ over $\bZ_{\ell}$.
\end{lem}
\begin{proof}
By taking a generating system $S$ of $\Gamma$ and associating the images of the elements of $S$, we can see that it is represented by a closed subscheme of $\wh{G}^S$. 
\end{proof}

Let $P_F$ be the wild inertia group defined as the maximal pro-$p$-subgroup of $I_F$. 
We take a lift $\wt{\sigma} \in W_F/P_F$ of $\sigma$ and a topological generator $t$ of $I_F/P_F$, and write 
$(W_F/P_F)^0$ for the subgroup of $W_F/P_F$ generated by $\wt{\sigma}$ and $t$. 
Let $W_F^0$ denote the inverse image of $(W_F/P_F)^0$ in $W_F$. 
Then we put 
\[
 Z^1 (W_F,\wh{G}) = \bigcup_{P} Z^1 (W_F^0/P,\wh{G}) , 
\]
where $P$ runs over open subgroups of $P_{F^*}$ that are normal subgroups of $W_F$. 
We can say that $W_F^0/P$ is a kind of discretization of $W_F/P$, and by using this discretization, it is reduced to the situation where no topology or condensed group appears\footnote{The idea of discretization has already appeared in \cite[\S 4]{HelCurBer}.}. 

We explain the difference between the L-parameters considered in $Z^1 (W_F,\wh{G})$ and the L-parameters in Definition \ref{dfn:SLLp}.
The condition of Definition \ref{dfn:SLLp} \ref{en:relev} is a condition to guarantee that each L-parameter comes from an element of $\Pi(G(F))$, so we do not consider it in the moduli\footnote{Since the L-group does not change even if $G$ is replaced by a quasi-split inner form, we can say that we only need to consider the case where $G$ is quasi-split, in which case all parabolic subgroups are relevant, and we may ignore that condition.}. 
Since we do not impose Frobenius semisimplicity in the moduli of the L-parameters, it is represented by a scheme\footnote{This is related to the fact that the set of semisimple element in an algebraic group is not necessarily locally closed.}. 
Also, in the above moduli, we are considering $\ell$-adic L-parameters instead of L-parameters of $\SL_2$-type. 
If we impose Frobenius semisimplicity, then there is a one-to-one correspondence between the equivalence class of L-parameters $\SL_2$-type in $\ol{\bQ}_{\ell}$-coefficients and the equivalence class of $\ell$-adic L-parameters in $\ol{\bQ}_{\ell}$-coefficients as in \cite[Proposition 1.13, Proposition 1.16]{ImaLLCell}, so either of them can be used when we consider the local Langlands correspondence. 
On the other hand, the are very different when we consider the moduli of the L-parameters: the moduli of $\ell$-adic L-parameters allows us to capture the change of the monodromy action continuously.
Moreover, if Frobenius semisimplicity is not imposed, there exist $\ell$-adic L-parameters that do not come from L-parameters of $\SL_2$-type even at $\ol{\bQ}_{\ell}$-valued points, as shown in \cite[Example 3.5]{BMIYJMmor}. 

We put $\mathrm{LP}_G =[Z^1 (W_F,\wh{G})/\wh{G}]$ as a quotient stack. 
If $G$ is unramified, we define $Z^1 (W_F,\wh{G})^{\mathrm{ur}}$ as an open and closed subscheme of $Z^1 (W_F,\wh{G})$ determined by the condition that 
\[
 I_{F,\mathrm{c}} \hookrightarrow W_{F,\mathrm{c}} \to \wh{G}(\Lambda_{\mathrm{c},\ell}) \rtimes W_{F,\mathrm{c}} \to \wh{G}(\Lambda_{\mathrm{c},\ell}) \rtimes (W_F/I_F)_{\mathrm{c}}
\]
is trivial, and put $\mathrm{LP}_G^{\mathrm{ur}} =[Z^1 (W_F,\wh{G})^{\mathrm{ur}}/\wh{G}]$. 
We write $\mathrm{IndCoh} (\mathrm{LP}_{G,\ol{\bQ}_{\ell}})$ for the Ind-completion of the derived category of coherent sheaves on 
$\mathrm{LP}_{G,\ol{\bQ}_{\ell}}$. 
For a locally profinite group $\mathsf{G}$, let $D(\mathsf{G},\Lambda)$ denote the derived category of smooth representations of $\mathsf{G}$ over $\Lambda$. 
The following conjecture is a kind of categorification of the local Langlands correspondence. 

\begin{conj}[{\cite[Conjecture 3.6]{HelderIH}, \cite[Conjecture 4.5.1]{ZhuCohLp}}\footnote{In \cite[Conjecture 4.5.1]{ZhuCohLp}, a more general conjecture is stated using the moduli space of L-parameters over $\bZ[1/p]$.}]\label{conj:LLCcat} 
We fix $c \in \ol{\bQ}_{\ell}$ such that $c^2=q$. 
Assume that $G$ is quasi-split, and fix a $G(F)$-conjugacy class $\fw$ of Whittaker data in $\bfC$-coefficients for $G$. 
Then there is a fully faithful functor 
	\[
	\fA_G \colon D (G(F),\ol{\bQ}_{\ell}) \to 
	\mathrm{IndCoh} (\mathrm{LP}_{G,\ol{\bQ}_{\ell}}) 
	\]
satisfying the following: 
\begin{enumerate}
\item 
For $(B,\psi) \in \fw$, we have 
$\fA_G (\cInd_{R_{\mathrm{u}}(B)(F)}^{G(F)} \psi ) \cong \cO_{\mathrm{LP}_{G,\ol{\bQ}_{\ell}}}$. 
\item 
If $G$ is unramified, then 
$\fA_G (\cInd_{K}^{G(F)}1) \cong \cO_{\mathrm{LP}^{\mathrm{ur}}_{G,\ol{\bQ}_{\ell}}}$ for $K \in \mathrm{HS}_{\fw}$. 
\end{enumerate}
\end{conj}

As for the Iwahori block part of Conjecture \ref{conj:LLCcat}, 
there is a result \cite{BCHNCohSpr} by Ben-Zvi--Chen--Helm--Nadler\footnote{A result by Hemo--Zhu is also announced in \cite{ZhuCohLp}.}. 

We explain the relation between Conjecture \ref{conj:LLCcat} and the Satake isomorphism assuming that $G$ is unramified. 
By associating the images of $q$-th powere Frobenius elements, 
we have $Z^1 (W_F,\wh{G})^{\mathrm{ur}} \cong \wh{G} \rtimes \sigma$.
Then we have 
\begin{align*}
 \End (\cInd_K^{G(F)} 1) \cong \cH (G(F),K) &\cong 
 \Gamma ((\wh{G}_{\ol{\bQ}_{\ell}} \rtimes \sigma)\sslash \wh{G}_{\ol{\bQ}_{\ell}},\cO) \\
 &\cong \Gamma (\wh{G}_{\ol{\bQ}_{\ell}} \rtimes \sigma,\cO)^{\wh{G}_{\ol{\bQ}_{\ell}}} \cong 
 H^0 \End (\cO_{\mathrm{LP}^{\mathrm{ur}}_{G,\ol{\bQ}_{\ell}}})
\end{align*}
by \eqref{eq:EndHiso} and \eqref{eq:HGammaiso}. 
Hence the Satake isomorphism means the degree $0$ part of the fully faithfulness 
\[
\End (\cInd_K^{G(F)} 1) \cong \End (\fA_G (\cInd_K^{G(F)} 1)) 
\]
of the functor $\fA_G$ in Conjecture \ref{conj:LLCcat}. 
Thus, although the Satake isomorphism is a classical result, we can say that it already implies a categorification of the local Langlands correspondence. 

The source of the functor $\fA_G$ in Conjecture \ref{conj:LLCcat} is a category of representations, but this is replaced by a geometric one to make an equivalence of categories in the geometric categorical local Langlands conjecture, which will be explained later. 

Also, in \cite{EGHintcatpL}, an analogue of Conjecture \ref{conj:LLCcat} in $p$-adic coefficients is formulated for $\GL_n$ using the Emerton--Gee stack (\cite{EmGeModstpg}), which is the moduli of etale $(\varphi,\Gamma)$-modules.

\section{Perfectoid space}

In this section, we explain the theory of perfectoid spaces. 
The basic reference is \cite{SchPerf}, but we follow \cite{FonPerBou} for the definition of perfectoid rings. 
The perfectoid space is defined as an adic space (\cite{Hugfs}) by Huber, so we first explain adic spaces.

\begin{dfn}
Let $R$ be a topological commutative ring. 
We say that $R$ is an f-adic ring if there exists a finitely generated ideal $I$ of an open subring of $R$ such that $\{ I^n \}_{n \geq 0}$ is a fundamental system of open neighbourhoods of $0$. 
If an f-adic ring $R$ has a topologically nilpotent invertible element, we say that $R$ is a Tate ring. 
\end{dfn}

Let $R$ be an f-adic ring. 
Let $\Cont (R)$ be the set of equivalence classes of continuous valuations on $R$, and we equip it with the topology generated by 
\[
 \left\{ \lvert \cdot \rvert \in \Cont (R) \setmid \lvert f \rvert \leq \lvert g \rvert \neq 0 \right\} \quad \quad (f,g \in R) . 
\]
A subset $S$ of $R$ is said to be bounded if for any open neighborhood $U$ of $0$, there exists some open neighborhood $V$ of $0$ such that $VS \subset U$. 
We put 
$R^{\circ}=\{ r \in R \mid \textrm{$\{ r^n \}_{n \geq 0}$ is bounded} \}$, $R^{\circ \circ}=\{ \textrm{topologically nilpotent elements of $R$} \}$. 
Then $R^{\circ}$ is an integrally closed, open subring of $R$. 

\begin{dfn}
A Huber pair is a pair $(R,R^+)$ of an f-adic ring $R$ and an open subring $R^+$ of $R$ such that $R^+$ is integrally closed in $R$ and $R^+ \subset R^{\circ}$. 
We say that a Huber pair $(R,R^+)$ is complete if $R$ is complete. 
\end{dfn}

For a complete Huber pair $(R,R^+)$, we put 
\[
 \Spa (R,R^+)=\left\{ \lvert \cdot \rvert \in \Cont (R) \setmid \textrm{$\lvert r \rvert \leq 1$ for any $r \in R^+$} \right\}
\]
and equip it with the topology induced from $\Cont (R)$. 
For a complete f-adic ring $R$, we write $\Spa (R)$ for $\Spa (R,R^{\circ})$. 

For $X=\Spa (R,R^+)$, we can construct presheaves $\cO_X$, $\cO_X^+$ of complete topological rings  over $X$ such that 
$\cO_X(X)=R$, $\cO_X^+(X)=R^+$ (\cf \cite[\S 1]{Hugfs}).
If $\cO_X$ is a sheaf, we say that $X$ is an affinoid adic space.
An adic space is defined by gluing affinoid adic spaces (\cf \cite[\S 2]{Hugfs}). 
We write $\lvert X \rvert$ for the underlying topological space of $X$.

\begin{dfn}
Let $X$ be an adic space. 
We say that $X$ is analytic if, for any $x \in X$, there exists an open affinoid adic space $U \subset X$ containing $x$ such that $\cO_X (U)$ is a Tate ring. 
\end{dfn}

\begin{dfn}
Let $f \colon X \to Y$ be a morphism of analytic adic spaces. 
We say that $f$ is a finite etale if, for any point $y$ of $Y$, there exists an affinoid open neighborhood $U$ of $y$ such that $V=f^{-1} (U)$ is an affinoid adic space, $\cO_V(V)$ is finite etale over $\cO_U(U)$ and $\cO^+_V(V)$ is the integral closure of $\cO^+_U(U)$. 
We say that $f$ is etale if, for any point $x$ of $X$, there exists an open neighborhood $U$ of $x$, an open neighborhood $V$ of $f(x)$ and a factorization $U \to W \to V$ of $f|_U$ such that $U \to W$ is an open immersion and $W \to V$ is finite etale. 
\end{dfn}

We can define the etale site $X_{\mathrm{et}}$ of an analytic adic space $X$ using coverings by etale morphisms (\cf \cite[Definition 8.2.19]{KeLiRpHF}).

\begin{dfn}
A complete Tate ring $R$ is said to be perfectoid if the following conditions are satisfied: 
\begin{enumerate}
\item 
$R^{\circ}/p \to R^{\circ}/p ;\ x \mapsto x^p$ is a surjection. 
\item 
$R^{\circ}$ is a bounded subset of $R$. 
\item 
Therre is a topologically nilpotent invertible element $\varpi$ of $R$ such that $p \in \varpi^p R^{\circ}$. 
\end{enumerate}
A perfectoid ring that is a non-archimedian field is called a perfectoid field. 
\end{dfn}

A Huber pair $(R,R^+)$ where $R$ is perfectoid is called a perfectoid Huber pair.
If $(R,R^+)$ is a perfectoid Huber pair, then $\Spa (R,R^+)$ is an affinoid adic space (\cite[Corollary 3.3.19]{KLRelpII}, \cite[Theorem 6.3]{SchPerf}), which is called an affinoid perfectoid space. 
An adic space that is locally isomorphic to an affinoid perfectoid space is called a perfectoid space. 
For an affinoid adic space $\Spa (R)$, let $\Perf_R$ denote 
the category of perfectoid spaces over $\Spa (R)$.

For a $p$-adic complete ring $A$, we put $A^{\flat}=\varprojlim_{x \mapsto x^p} A$ as monoids with respect to the product. Then the natural map $A^{\flat} \to \varprojlim_{x \mapsto x^p} A/pA$ is is an isomorphism of monoids (\cf \cite[Proposition 2.1.2]{FaFoCfv}).
Furthermore, we define 
\begin{equation}\label{eq:flatsum}
 x+y=\left(\lim_{n \mapsto \infty} (x_{i-n}+y_{i-n})^{p^n} \right)_{i \leq 0}
\end{equation}
for $x=(x_i)_{i \leq 0}, y=(y_i)_{i \leq 0} \in A^{\flat}$. 
Then $A^{\flat} \to \varprojlim_{x \mapsto x^p} A/pA$ is an isomorphism of perfect rings of characteristic $p$ (\cf \cite[Corollaire 2.1.4]{FaFoCfv}).

Let $(R,R^+)$ be a perfectoid Huber pair. 
We put 
$R^{\flat}=\varprojlim_{x \mapsto x^p} R$ and 
$R^{+,\flat}=\varprojlim_{x \mapsto x^p} R^+$ 
as topological monoids with respect to the product, and define addition as \eqref{eq:flatsum}. 
Then $(R^{\flat}, R^{+,\flat})$ is a perfectoid Huber pair of characteristic $p$. 
We say that $(R^{\flat}, R^{+,\flat})$ is a tilt of $(R, R^+)$. 
By this, we have the tilting functor 
\[
 \flat \colon \Perf_{\cO_F} \to \Perf_{k_F} ;\ 
 \Spa (R,R^+) \mapsto \Spa (R^{\flat}, R^{+,\flat}) . 
\]

\begin{dfn}
For $S \in \Perf_{k_F}$, an untilt of $S$ over $\cO_F$ is a pair of 
$S^{\sharp} \in \Perf_{\cO_F}$ and an isomorphism 
$(S^{\sharp})^{\flat} \cong S$ in $\Perf_{k_F}$. 
\end{dfn}

In the following, we explain a description untilts. 
We define the functor 
\begin{align*}
 W_{\cO_F} \colon 
 (\textrm{category of perfect rings over $k_{F}$}) 
 \to  
  (\textrm{category of $\fm_F$-adic complete rings over $\cO_F$}) 
\end{align*}
by 
\[
W_{\cO_F} (R)= 
\begin{cases}
	W(R) \otimes_{W(k_F)} \cO_F & (\textrm{if the characteristic of $F$ is $0$})\\ 
	R \wh{\otimes}_{k_F} \cO_F & (\textrm{if the characteristic of $F$ is $p$.}) 
\end{cases}
\]
Then $W_{\cO_F}$ is a left adjoint functor of the functor 
\begin{align*}
 \flat \colon 
 &(\textrm{category of $\fm_F$-adic complete rings over $\cO_F$}) \\ 
 &\to 
 (\textrm{category of perfect rings over $k_{F}$});\ A \mapsto A^{\flat}
\end{align*}
(\cf \cite[Proposition 2.1.7]{FaFoCfv}). 
Hence, for an $\fm_F$-adic complete ring $A$ over $\cO_F$, 
there is the adjoint morphism $\theta_A \colon W_{\cO_F}(A^{\flat}) \to A$. 

\begin{dfn}
Let $(R,R^+)$ be a perfectoid Huber pair over $k_F$. An element of 
\[
[R^{\circ\circ}] +(\fm_F -\fm_F^2) [(R^+)^{\times}] +\fm_F^2 W_{\cO_F}(R^+) 
\]
is called a primitive element of degree $1$ in $W_{\cO_F}(R^+)$. 
An ideal of $W_{\cO_F}(R^+)$ generated by a single primitive element of degree $1$ is called a primitive ideal of degree $1$ in $\cO_F$-coefficients. 
\end{dfn}

\begin{thm}[{\cite[Theorem 3.3.8]{KLRelpII}}]\label{thm:untpri}
The functor 
\[
 (R,R^+) \mapsto (R^{\flat},R^{+,\flat}, \Ker \theta_{R^+}) 
\]
from the category of perfectoid Huber pairs over $\cO_F$ to the 
category of pairs of 
a perfectoid Huber pair over $k_F$ and a primitive ideal of degree $1$ in $\cO_F$-coefficients 
gives an equivalence of categroies, whose quasi-inverse functor is given by 
\[
 (R,R^+,I) \mapsto ((W_{\cO_F}(R^+)/I)[[\varpi_R]^{-1}],W_{\cO_F}(R^+)/I), 
\]
where $\varpi_R$ is a topologically nilpotent invertible element of $R$. 
\end{thm}

\begin{dfn}
Let $X$ be a perfectoid space. 
We say that $X$ is strictly totally disconnected if $X$ is quasi-compact, quasi-separated and any etale covering of $X$ splits. 
\end{dfn}

The following proposition will be used later in Definition \ref{dfn:pefectconst}. 

\begin{prop}\label{prop:stdtopos} 
Let $X$ be a strictly totally disconnected perfectoid space.
If $X'$ is a perfectoid space which is etale over $X$, then there exists an open covering $\{ U_i \}_{i \in I}$ of $X'$ such that $U_i \to X$ is an open immersion for any $i \in I$. 
Thus, the etale topos of $X$ and the topos of $\lvert X \rvert$ are equivalent.
\end{prop}
\begin{proof}
It suffices to show the first claim. 
By writing the etale morphism $f \colon X' \to X$ as a composition of an open embedding $X'\to Y$ and a finite etale morphism $Y \to X$ locally on $X$ and $X'$, and replacing $X'$ by $Y$, we may assume that $f$ is finite etale. 
We may further replace $X$ by $\Image f$ and assume that $f$ is a finite etale covering. 
Since $f$ splits, we can write $X'=X \amalg X_1'$ and $X_1' \to X$ is a finite etale morphism. 
We can replace $X$ with the image of $X_1'\to X$ and repeat the same process. 
\end{proof}

\section{Diamond and v-stack}

\begin{dfn}
We say that a morphism $\Spa (B,B^+)\to \Spa (A,A^+)$ of affinoid perfectoid spaces is affinoid pro-etale if there is a directed inverse system $\{ \Spa (A_i,A_i^+)\}_{i \in I}$ of affinoid perfectoid spaces which are etale over $\Spa(A,A^+)$, and $(B,B^+)$ is isomorphic to the completion of $\mathrm{colim}_{i \in I} (A_i,A_i^+)$.
We say that a morphism $f \colon Y \to X$ of perfectoid spaces is pro-etale if for any $y \in Y$, there exists an affinoid open neighborhood $U$ of $f(y)$ and an affinoid open neighborhood $V \subset f^{-1}(U)$ of $y$ such that $f|_V \colon U \to V$ is affinoid pro-etale. 
\end{dfn}

\begin{dfn}\label{dfn:Ccov}
Let $\sC$ be a class of morphisms of perfectoid spaces. 
We say that a family $\{ f_i \colon Y_i \to X \}_{i \in I}$ of morphisms of perfectoid spaces is $\sC$-covering if the following holds: 
\begin{enumerate}
\item 
For any $i \in I$, we have $f_i \in \sC$. 
\item\label{en:cocfin}
For any quasi-compact open subset $U$ of $X$, there exists a finite subset $J \subset I$ and a family of quasi-compact open subsets $\{ V_j \subset Y_j \}_{j \in J}$ such that $U=\bigcup_{j \in J} f_j (V_j)$. 
\end{enumerate}
\end{dfn}

If we consider all the pro-etale morphisms as $\sC$ in Definition \ref{dfn:Ccov}, $\sC$-covering is called pro-etale covering
\footnote{In Definition \ref{dfn:Ccov}, if all morphisms in $\sC$ are open maps, then the condition of \ref{en:cocfin} can be replaced by $X=\bigcup_{i \in I} f_i (Y_i)$. However, since pro-etale morphisms are not necessarily open maps, we impose the condition \ref{en:cocfin}.}. 
Thus we can define a pro-etale sheaf.
By \cite[Corollary 8.6]{SchEtdia}, 
\[
 h_X \colon \Perf_{\cO_F}^{\mathrm{op}} \to \mathrm{Sets} ;\ 
 Y \mapsto \Hom (Y,X) 
\]
is a pro-etale sheaf for $X \in \Perf_{\cO_F}$. 
Let $k$ be an algebraic field extension of $k_F$. 

\begin{dfn}\label{dfn:dia}
We say that a pro-etale sheaf $\cD$ on $\Perf_k$ 
is a diamond over $k$ if there are $X,R \in \Perf_k$ and pro-etale morphisms $s,t \colon R \to X$ such that $h_{(s,t)} \colon h_R \to h_X \times h_X $ gives an equivalence relation and $\cD \cong h_X /h_R$. 
\end{dfn}

For $X \in \Perf_k$, the pro-etale sheaf $h_X$ is a diamond, for which we simply write $X$ in the following. 

If we consider all the morphisms of perfectoid spaces as $\sC$ in Definition \ref{dfn:Ccov}, $\sC$-covering is called v-covering. 
Thus we can define a v-sheaf.

\begin{ex}
Let $T$ be a topological space. 
We write $\ul{T}$ for the functor sending $S \in \Perf_k$ to $C^0(\lvert S \rvert,T)$. 
By \cite[Lemma 1.1.10]{HubEtadic} and \cite[Lemma 2.5]{SchEtdia}, $\ul{T}$ is a v-shaef.
\end{ex}

A stack with respect to the topology given by v-covering on $\Perf_k$ is called a v-stack over $k$. 

\begin{dfn}
we say that 
a morphism $f \colon X \to Y$ of v-stacks over $k$ is an open immersion if $X \times_Y Y' \to Y'$ is an open immersion in $\Perf_k$ for any $Y' \in \Perf_k$ and $Y' \to Y$. 
\end{dfn}

For v-shaef and v-stack, we consider the following notions given by set-theoretic conditions\footnote{These concepts make substantial sense only when we put set-theoretic conditions like \cite[\S 4]{SchEtdia} on perfectoid spaces, but here we give definitions to adjust terminologies to \cite{FaScGeomLLC} and \cite{SchEtdia}.}. 

\begin{dfn}
We say that a v-shaef $X$ on $\Perf_k$ is a small v-shaef if there exist some $Y \in \Perf_k$ and a surjection $Y \to X$ of v-shaeves. 
\end{dfn}

By \cite[Proposition 11.9]{SchEtdia} and Definition \ref{dfn:dia}, a diamond over $k$ is a small v-shaef on $\Perf_k$. 
Also, as we can see from the definition, any open sub-v-shaef of a diamond is a diamond\footnote{Actually, more strongly, it is known by \cite[Proposition 11.10]{SchEtdia} that any sub-v-shaef of a diamond is a diamond.}.

\begin{dfn}\label{dfn:smvst}
We say that 
a v-stack $X$ over $k$ is a small v-stack if there exist some $Y \in \Perf_k$ and a surjection $Y \to X$ of v-stacks such that $Y \times_X Y$ is a small v-shaef. 
\end{dfn}

Let $X$ be a small v-stack over $k$. We put 
\[
 \lvert X \rvert = \{ \Spa (K,K^+) \to X \}/{\sim} , 
\]
where $(K,K^+)$ is a Huber pair such that $K$ is a perfectoid field over $k$, and $\sim$ is an equivalence relation generated by relations that $\Spa (K,K^+) \to X$ and $\Spa (K',K'^+) \to \Spa (K,K^+) \to X$ are equivalent 
if $\Spa (K',K'^+) \to \Spa (K,K^+)$ is surjective. 
If we take $Y$ as Definition \ref{dfn:smvst}, 
$R \in \Perf_k$ and a surjection $R \to Y \times_X Y$ of v-shaeves, 
then there is a natural bijection 
$\lvert Y \rvert / \lvert R \rvert \cong \lvert X \rvert$ and 
the topology on $\lvert X \rvert$ induced by this bijection 
is independent of $Y$, $R$ (\cf \cite[Proposition 12.7]{SchEtdia}). 
We call $\lvert X \rvert$ with this topology the underlying topological space of $X$. 
Associating a morphism from $S \in \Perf_k$ to $X$ with the induced morphism  $\lvert S\rvert  \to \lvert X \rvert$, we have 
\begin{equation}\label{eq:vsttop}
X \to \ul{\lvert X \rvert} . 
\end{equation}

\begin{dfn}
We say that a morphism $f \colon X \to Y$ of v-stacks over $k$ is a closed immersion if $X \times_Y Y' \to Y'$ is a quasi-compact\footnote{See \cite[VI D\'{e}finition 1.1, D\'{e}finition 1.7]{SGA4-2} for definitions of sheaves and morphisms between them being quasi-compact.} injection of v-shaeves and $\lvert X \times_Y Y' \rvert \to \lvert Y' \rvert$ is a closed immersion for any small v-shaef $Y'$ on $\Perf_k$ and $Y' \to Y$\footnote{This definition is equivalent to \cite[Definition 10.7 (ii)]{SchEtdia} by \cite[Corollary 10.6]{SchEtdia}. See also \cite[Definition 17.4.2]{ScWeBLp} for a definition in the case of v-shaeves.}. 
\end{dfn}

\begin{dfn}
We say that a morphism $f \colon X \to Y$ of v-stacks over $k$ is separated if the diagonal morphism $\Delta_f \colon X \to X \times_Y X$ is a closed immersion. 
\end{dfn}

\begin{dfn}
Let $X$ be a v-shaef on $\Perf_k$. 
If $X$ is quasi-compact, quasi-separated\footnote{See \cite[VI D\'{e}finition 1.13]{SGA4-2} for a definition of sheaves being quasi-separated.} and $\{ \lvert U \rvert \}_{U \subset X}$ gives a open basis of $\lvert X \rvert$ 
when $U$ runs through quasi-compact open sub-v-shaeves of $X$, 
then $X$ is said to be spatial. 
We say that $X$ is locally spatial if $X$ has an open cover by spatial v-shaeves. 
\end{dfn}

\begin{dfn}
Let $f \colon X \to Y$ be a locally separated morphism\footnote{we say that $f$ is locally separated if it is separated on some open covering of $X$. If $X \in \Perf_k$, then $f$ is always locally separated.} of diamonds. 
We say that $f$ is etale if $X \times_Y Y' \in \Perf_k$ and $X \times_Y Y' \to Y'$ is etale for any $Y' \in \Perf_k$ and $Y' \to Y$. 
\end{dfn}

Using coverings by etale morphism, we can define an etale site $X_{\mathrm{et}}$ of a locally spatial diamond $X$ (\cite[Definition 14.1]{SchEtdia}). 

For an adic space $X$ over $\Spa (\cO_F)$, we define $X^{\diamondsuit} \colon \Perf_{k_F}^{\mathrm{op}} \to \mathrm{Sets}$ by 
\[
X^{\diamondsuit} (S)=\{ (S^{\sharp},(S^{\sharp})^{\flat} \cong S,S^{\sharp} \to X) \}, 
\]
where 
$(S^{\sharp},(S^{\sharp})^{\flat} \cong S)$ is an untilt of $S$ over $\cO_F$, and $S^{\sharp} \to X$ is a morphism over $\Spa (\cO_F)$. 

\begin{thm}[{\cite[Lemma 15.6]{SchEtdia}}]
If $X$ is an analytic adic space over $\Spa (\cO_F)$, then $X^{\diamondsuit}$ is a locally spatial diamond over $k_F$ and $\lvert X \rvert \cong \lvert X^{\diamondsuit} \rvert$. 
Furthermore, the functor $\diamondsuit$ induces an equivalence $X_{\mathrm{et}} \cong X^{\diamondsuit}_{\mathrm{et}}$ of sites. 
\end{thm}

\begin{ex}
Let $X=\Spa (\bQ_p \langle T \rangle ,\bZ_p \langle T \rangle )$. 
Then $\wt{X}=\Spa (\bQ_p^{\mathrm{cyc}} \langle T^{1/p^{\infty}} \rangle ,\bZ_p^{\mathrm{cyc}} \langle T^{1/p^{\infty}} \rangle )$ is a perfectoid space which is a $(\bZ_ p^{\times} \times \bZ_p)$-covering of $X$, and we have $X^{\diamondsuit} \cong \wt{X}^{\flat}/\ul{\bZ_p^{\times} \times \bZ_p}$,
where $\bZ_p^{\mathrm{cyc}}$ is the $p$-adic completion of $\bZ_p(\mu_{p^{\infty}})$, $\bQ_p^{\mathrm{cyc}}=\bZ_p^{\mathrm{cyc}}[1/p]$ and the action of $\bZ_p^{\times}$ and $\bZ_p$ on $\wt{X}$ is given by the natural actions on $\mu_{p^{\infty}}$ and $T^{1/p^{\infty}}$, respectively. 
\end{ex}

\begin{dfn}
Let $f \colon X \to Y$ be a locally separated morphism of diamonds. 
We say that $f$ is quasi-pro-etale if 
$X \times_Y Y' \in \Perf_k$ and $X \times_Y Y' \to Y'$ is pro-etale for any strictly totally disconnected perfectoid space $Y'$ and $Y' \to Y$. 
\end{dfn}

If we take a diamond $\cD$ and $X$, $R$ as in Definition \ref{dfn:dia}, the natural morphism $X \to \cD$ is quasi-pro-etale (\cf \cite[Proposition 11.3]{FaScGeomLLC}). 
We can also define a quasi-pro-etale site $X_{\mathrm{qproet}}$ of a diamond $X$ (\cf \cite[Definition 14.1]{SchEtdia}).

\begin{dfn}
Let $X$ be a small v-stack over $k$.
We define $X_{\mathrm{v}}$ to be the site whose objects are small v-shaeves $Y$ over $X$ and coverings are families $\{ Y_i \to Y \}_{i \in I}$ of morphisms over $X$ such that $\bigcup_{i \in I} Y_i \to Y$ is a surjection of v-shaeves.  
\end{dfn}

Let $\Lambda$ be a ring. 
For a topos $T$, let $D(T,\Lambda)$ denote the derived category of sheaves of $\Lambda$-modules on $T$.

\begin{prop}[{\cite[Proposition 14.10]{SchEtdia}}]
If $X \in \Perf_k$ is strictly totally disconnected, then $D(X_{\textrm{et}},\Lambda )\to D(X_{\mathrm{v}} ,\Lambda )$ given by pullback is fully faithful. 
\end{prop}

\begin{dfn}
Let $X$ be a small v-stack. 
We define $D_{\mathrm{et}}(X,\Lambda)$ as the full subcategory of $D(X_{\mathrm{v}},\Lambda)$ consisting of every object $A$ of $D(X_{\mathrm{v}},\Lambda)$ such that $f^* A \in D(Y_{\mathrm{v}},\Lambda)$ is in the essential image of $D(Y_{\textrm{et}},\Lambda )$ for any strictly totally disconnected $Y \in \Perf_k$ and $f \colon Y \to X$. 
\end{dfn}

If $X$ is a strictly totally disconnected perfectoid space, then $D_{\textrm{et}}(X,\Lambda )\cong D(X_{\textrm{et}},\Lambda )$  as seen from the definition, and we have $D(X_{\textrm{et}},\Lambda )\cong D(\lvert X \rvert,\Lambda )$ by Proposition \ref{prop:stdtopos}. 

\begin{dfn}\label{dfn:pefectconst}
Let $X$ be a small v-stack. 
We say that $A \in D_{\mathrm{et}}(X,\Lambda )$ is perfectly constructible if, for any strictly totally disconnected $Y \in \Perf_k$ and $f \colon Y \to X$, there is a decomposition $\coprod_{i \in I} S_i$ of $\lvert Y \rvert$ as sets by finite families of constructible locally closed sets such that the image of $f^* A$ under $D_{\mathrm{et}}(Y,\Lambda ) \cong D(Y_{\mathrm{et}},\Lambda ) \cong D(\lvert Y \rvert,\Lambda )$ is isomorphic to a complex of constant sheaves associated with a perfect complex of $\Lambda$-modules on each $S_i$.  
\end{dfn}

Assume that there exists a positive integer $n$ prime to $p$ such that $n \Lambda=0$. 
Then there is a six functor formalism for $D_{\mathrm{et}}(-,\Lambda )$ (\cf \cite[\S 17, \S 22, \S 23]{SchEtdia}). 

\begin{dfn}
Let $f \colon X \to S$ a morphism of locally spatial diamonds which is compactifiable (\cf \cite[Definition 22.2]{SchEtdia}) and $\mathrm{dim}. \trg f < \infty$ locally on $S$ (\cf \cite[Definition 21.7]{SchEtdia}). 
We say that $A \in D_{\mathrm{et}}(X,\Lambda )$ is over locally acyclic over $S$ is the following conditions are satisfied: 
\begin{enumerate}
\item 
For any geometric point $\ol{s}$ of $S$, geometric point $\ol{x}$ of $X$ over $\ol{s}$ and geometric point $\ol{t}$ of $S$ that is a generalization of $\ol{s}$, the natural morphism $R\Gamma (X_{\ol{x}} ,A) \to R \Gamma (X_{\ol{x}} \times_{S_{\ol{s}}} S_{\ol{t}},A)$ is an isomorphism. 
\item 
For any separated etale morphism $j \colon U \to X$ such that $f \circ j$ is quasi-compact, $R(f \circ j)_! j^*A \in D_{\mathrm{et}}(S,\Lambda)$ is perfectly constructible. 
\end{enumerate}
\end{dfn}

\begin{dfn}\label{dfn:ula}
Let $f \colon X \to S$ be a morphism of small v-stacks which is compactifiable, $\mathrm{dim}. \trg f < \infty$ locally on $S$ and representable by locally spatial diamonds\footnote{This means that $X \times_S S'$ is a locally spatial diamond for any locally spatial diamond $S'$ and morphism $S' \to S$.}, and let $A \in D_{\mathrm{et}}(X,\Lambda )$. 
We say that $A$ is universally locally acyclic over $S$ if the pullback $A'\in D_{\mathrm{et}}(X \times_S S',\Lambda )$ of $A$ is locally acyclic over $S'$ for any locally spatial diamond $S'$ and $S'\to S$. 
\end{dfn}

If $f \colon X \to S$ is a morphism satisfying the condition in Definition \ref{dfn:ula}, we put 
\[
 A^{\vee}=R\! \sHom_{\Lambda} (A,\Lambda), \quad 
\bD_{X/S} (A) = R\! \sHom_{\Lambda} (A,Rf^! \Lambda) 
\]
for $A \in D_{\mathrm{et}}(X,\Lambda )$, and call them 
the dual of $A$ and the relative Verdier dual of $A$ with respect to $f$, respectively. 
If $A$ is universally locally acyclic over $S$, then so is $\bD_{X/S}(A)$, and we have $A \cong \bD_{X/S}(\bD_{X/S}(A))$ (\cf \cite[Corollary IV.2.25]{FaScGeomLLC}). 

\begin{dfn}\label{def:lcohsm}
Let $f \colon X \to S$ be a morphism of small v-stacks which is compactifiable, $\mathrm{dim}. \trg f < \infty$ locally on $S$ and representable by locally spatial diamonds. 
We say that $f$ is $\ell$-cohomologically smooth if 
$\bF_{\ell} \in D_{\mathrm{et}}(X,\bF_{\ell})$ is universally locally acyclic over $S$ and $Rf^! \bF_{\ell}$ is invertible\footnote{This definition is equivalent to \cite[Definition 23.8]{SchEtdia} by \cite[Proposition IV.2.33]{FaScGeomLLC}.}.
\end{dfn}

\begin{dfn}
A small v-stack $X$ is called an Artin v-stack if it satisfies the following conditions: 
\begin{enumerate}
\item 
The diagonal morphism $\Delta_X \colon X \to X \times X$ is representable by locally spatial diamonds. 
\item 
There are a locally spatial diamond $U$ and an $\ell$-cohomologically smooth surjection $U \to X$. 
\end{enumerate}
\end{dfn}

The notion of $\ell$-cohomologically smooth morphism is also defined for morphisms of Artin v-stacks that are not necessarily representable by locally spatial diamond as follows\footnote{We can see that this definition is consistent with Definition \ref{def:lcohsm} by \cite[Proposition 23.13]{SchEtdia}.}. 

\begin{dfn}
Let $f \colon X \to Y$ be a morphism of Artin v-stacks. 
We say that $f$ is $\ell$-cohomologically smooth if there exist a locally spatial diamond $U$ and an $\ell$-cohomologically smooth surjection $g \colon U \to X$ such that $f \circ g \colon U \to Y$ is $\ell$-cohomologically smooth. 
\end{dfn}

\section{Moduli space of $G$-bundles over the Fargues--Fontaine curve}

Fargues--Fontaine curves are curves introduced in the study of $p$-adic Galois representations in \cite{FaFoCfv}, and are defined as a scheme and as an adic space. 
Here, we define the moduli space of $G$-bundles over Fargues--Fontaine curves using a relative version of Fargues--Fontaine curves as adic spaces.

In the following, $\varphi$ denotes the $q$-th power Frobenius map on a ring of characteristic $p$. 
For an affinoid perfectoid space $S=\Spa (R,R^+)$ over $\Spa (k_F)$, we put 
\[
 Y_S =\left\{ \lvert \cdot \rvert \in \Spa (W_{\cO_F}(R^+)) \setmid 
 \lvert \pi_F[\varpi_R] \rvert \neq 0 \right\}, \quad X_S=Y_S / \varphi^{\bZ}, 
\]
where $\pi_F$ is a uniformizer of $F$, $\varpi_R$ is a topologically nilpotent invertible element of $R$, we equip $W_{\cO_F}(R^+)$ with $(\pi_F,[\varpi_R])$-adic topology and $\varphi$ acts on $Y_S$ via the action on $R^+$. 
They glue together and gives adic spaces $Y_S$ and $X_S$ over $F$ for any $S \in \Perf_{k_F}$, where $X_S$ is called the relative Fargues--Fontaine curve for $S$. 

We put $\Spd F =\Spa (F,\cO_F)^{\diamondsuit}$. 
We define the action of $\varphi$ on $\Spd F$ by 
\[
 (S^{\sharp},(S^{\sharp})^{\flat} \cong S,S^{\sharp} \to \Spa (F,\cO_F) ) \mapsto 
 (S^{\sharp},(S^{\sharp})^{\flat} \cong S \stackrel{\varphi}{\cong} S,S^{\sharp} \to \Spa (F,\cO_F) ). 
\]
Let $d$ be a nonnegative integer, and let $\Sigma_d$ denote the $d$-th symmetric group. 
We put $\Div_{X,k_F}^d =(\Spd F/\varphi^{\bZ})^d/\Sigma_d$. 
As a consequence of Theorem \ref{thm:untpri}, if $S \in \Perf_{k_F}$, we can associate an element of $\Div_{X,k_F}^d(S)$ with a closed Cartier divisor on $X_S$ (\cf \cite[Definition 5.3.7]{ScWeBLp}), and the Cartier divisor thus obtained is called a closed Cartier divisor of degree $d$ on $X_S$\footnote{We use this terminology because a closed Cartier divisor of degree $d$ is written as the sum of $d$ closed Cartier divisors of degree $1$ v-locally on $S$, and a closed Cartier divisor of degree $1$ is determined from a primitive ideal of degree $1$ in $\cO_F$-coefficients.}. 

In the following, let $k$ be the residue field of $\breve{F}$ and $\Div_{X}^d=\Div_{X,k_F}^d \otimes_{k_F} k$. 
As in \cite[Proposition VI.9.2]{FaScGeomLLC}, a local system on $\Div_{X}^d$ corresponds to a continuous representation of $W_F^d$, so $\Div_{X}^1$ plays the role of an algebraic curve in the geometric Langlands correspondence. 

\begin{rem}\label{rem:algvsFF}
When $X$ is an algebraic curve over a field, $X$ itself can be regarded as a moduli space of effective Cartier divisors of degree $1$ on $X$ by considering sections of the structural morphism.
In the case of Fargues--Fontaine curves, for $S \in \Perf_{k_F}$, there is no structural morphism from $X_S$ to $S$, and we cannot consider $S$-valued points in $X_S$, so the situation is very different and $\Div_{X}^1$ appears. 
\end{rem}

Let $b \in G(\breve{F})$. 
For $S \in \Perf_k$, we define the $G$-bundle $\cE_{b,X_S}$ over $X_S$ by $G_{\breve{F}} \times_{\breve{F}} Y_S /(b\sigma \times \varphi )^{\bZ}$. 
We define the v-shaef $\wt{G}_b$ on $\Perf_k$ by $\wt{G}_b(S)=\Aut (\cE_{b,X_S})$. 
We define $\Bun_G \colon \Perf_{k_F} \to \mathrm{Groupoids}$ by 
\[ \Bun_G(S)=(\textrm{the groupoid of $G$-bundles over $X_S$}). \]
Then $\Bun_G$ is an $\ell$-cohomologically smooth Artin v-stack over $\Spd k$ (\cf \cite[Proposition IV. 1.19]{FaScGeomLLC}). 
\begin{thm}[{\cite[Theorem 10]{AnsRedFF}, \cite[Th\'{e}or\`{e}me 5.1]{FarGtor}}]\label{thm:BGbun}
Let $(C,C^+)$ be a Huber pair such that $C$ is an algebraically closed perfectoid field over $k$. 
Then $b \mapsto \cE_{b,X_{\Spa(C,C^+)}}$ induces a bijection $B(G) \stackrel{\sim}{\longrightarrow} \Bun_G (\Spa(C,C^+))/{\sim}$. 
\end{thm}

By the bijection in Theorem \ref{thm:BGbun}, we have a bijection $B(G) \stackrel{\sim}{\longrightarrow} \lvert \Bun_G \rvert$. 
For $[b] \in B(G)$, we put 
\[
 \Bun_G^{[b]}=\Bun_G \times_{\ul{\lvert \Bun_G \rvert}} \ul{\{[b]\}}
\]
using the morphism of \eqref{eq:vsttop}. 
Let $*=\Spd k$. 
By \cite[Theorem III.0.2]{FaScGeomLLC}, the natural morphism $i^{[b]} \colon \Bun_G^{[b]} \to \Bun_G$ is a locally closed immersion and we have 
\begin{equation}\label{eq:Bunquot}
\Bun_G^{[b]} \cong [*/\wt{G}_b] . 
\end{equation}
There exists a natural isomorphism $\pi_0 (\lvert \wt{G}_b \rvert) \cong G_b(F)$, which and \eqref{eq:vsttop} give 
\begin{equation}\label{eq:wtGbul}
\wt{G}_b \to \ul{\lvert \wt{G}_b \rvert} \to \ul{G_b(F)}. 
\end{equation}
\eqref{eq:wtGbul} is a surjection, and is an isomorphism if $b$ is a basic element (\cf \cite[Proposition III.5.1]{FaScGeomLLC}). 
By \eqref{eq:Bunquot} and \eqref{eq:wtGbul}, there is a natural morphism 
\begin{equation}\label{eq:Buntoquot}
 \Bun_G^{[b]} \to [*/\ul{G_b(F)}]. 
\end{equation}

\section{Geometric Satake equivalence}
For a perfectoid space $S$ over $\Div_{X}^d$, we write $D_S$ for the corresponding closed Cartier divisor on $X_S$, and $\cI_S$ for its invertible ideal sheaf. 
By shrinking $S$, we can make $D_S$ an affinoid adic space. 
For an affinoid perfectoid space $S$ over $\Div_{X}^d$ such that 
$D_S$ is an affinoid adic space, we put 
\begin{align*}
B_{\Div_X^d}^+ (S)&= \{ \textrm{the global sections of the completion of $\cO_{X_S}$ along $\cI_S$}\}, \\ 
B_{\Div_X^d} (S)&=B_{\Div_X^d}^+ (S) \left[ \frac{1}{I_S} \right] , 
\end{align*}
where $I_S$ denotes the ideal of $B_{\Div_X^d}^+ (S)$ determined by $\cI_S$. 
By this, we obtain v-shaeves $B_{\Div_X^d}^+$ and $B_{\Div_X^d}$ over $\Div_X^d$. 
For an affinoid perfectoid space $S$ over $\Div_{X}^d$ such that 
$D_S$ is an affinoid adic space, 
we define $\cHck_{G,\Div_X^d} (S)$ by 
\begin{align*}
\{ &\textrm{the groupoid of $G$-bundles $\cE_1$, $\cE_2$ over $B_{\Div_X^d}^+(S)$} \\ &\textrm{and an isomorphism between $\cE_1$ and $\cE_2$ over $B_{\Div_X^d} (S)$}\}, 
\end{align*}
and $\Gr_{G,\Div_X^d} (S)$ by 
\[
\{ \textrm{the groupoid of a $G$-bundle $\cE$ over $B_{\Div_X^d}^+(S)$ and a trivialization of  $\cE$ over $B_{\Div_X^d} (S)$}\}. 
\]
By this, we obtaine small v-stacks $\cHck_{G,\Div_X^d}$ and $\Gr_{G,\Div_X^d}$ over $\Div_X^d$, 
and $\Gr_{G,\Div_X^d}$ is equivalent to the v-shaef obtiane by taking isomorphism classes (\cf \cite[Proposition VI.1.7, Proposition VI.1.9]{FaScGeomLLC}). 
We call $\cHck_{G,\Div_X^d}$ a local Hecke stack, and $\Gr_{G,\Div_X^d}$ a Beilinson--Drinfeld Grassmann variety. 
By definition, we have a natual morphism $\Gr_{G,\Div_X^d} \to \cHck_{G,\Div_X^d}$. 
By interchanging $\cE_1$ and $\cE_2$ that appear in the definition of $\cHck_{G,\Div_X^d}$, 
we obtain an isomorphism 
\begin{equation}\label{eq:swHck}
 \mathrm{sw} \colon \cHck_{G,\Div_X^d} \stackrel{\sim}{\lra}
 \cHck_{G,\Div_X^d} . 
\end{equation}
For a small v-stack $S$ over $\Div_X^d$, we put 
\[
 \Gr_{G,S/\Div_X^d} = \Gr_{G,\Div_X^d} \times_{\Div_X^d} S, \quad 
 \cHck_{G,S/\Div_X^d} =\cHck_{G,\Div_X^d} \times_{\Div_X^d} S, 
\]
and define v-shaeves $L^+_S G$, $L_S G$ over $S$ by putting 
\[
 (L^+_S G)(S')=G(B^+_{\Div_X^d} (S')), \quad 
 (L_S G)(S')=G(B_{\Div_X^d} (S'))
\]
for a perfectoid space $S'$ over $S$. 

\begin{rem}
Let $F'$ be a finite subextension of $F^{\mathrm{sep}}/F$. 
For $S \in \Perf_k$, 
we have $X'_S \cong X_S \otimes_F F'$ writing $X'_S$ for the relative Fargues--Fontaine curve for $S$ determined from $F'$. 
Let $\Div_{X'}^d$ be the v-shaef of closed Cartier divisors of degree $d$ determined from $F'$.
For an affinoid perfectoid space $S$ over $\Div_{X'}^d$ such that $D_S$ is an affinoid adic space, the rings $B_{\Div_X^d}^+ (S)$ over $F$ and $B_{\Div_{X'}^d}^+ (S)$ over $F'$ are isomorphic as rings over $F$, so there is a correspondence between $G$-bundles over $B_{\Div_X^d}^+ (S)$ and $G_{F'}$-bundles over $B_{\Div_{X'}^d}^+ (S)$. 
Thus $\cHck_{G,S/\Div_X^d} \cong \cHck_{G_{F'},S/\Div_{X'}^d}$ for a small v-stack $S$ over $\Div_{X'}^d$.
When we consider v-locally on $S$ or geometric pointwisely, we can use this isomorphism to reduce it to the case where $G$ splits. 
\end{rem}

Let $(\bX ,\Phi,\Delta, \check{\bX} ,\check{\Phi},\check{\Delta})=\mathrm{BR} (G_{F^{\mathrm{sep}}})$. We put 
\[
 \check{\bX}^+ =\{ \mu \in \check{\bX} \mid \textrm{$\langle \mu , \alpha \rangle \geq 0$ for any $\alpha \in \Delta$} \}. 
\]
We assume that $G$ is split for a while. 
For a geometric point $\ol{s}$ of $\Div_X^d$, we have 
\begin{equation}\label{eq:Hckdc}
	\cHck_{G,\Div_X^d} (\ol{s}) \cong \left[ \bigl( L^+_{\Div_X^d} G \bigr) (\ol{s}) \middle\backslash \bigl( L_{\Div_X^d} G \bigr) (\ol{s}) \middle/ \bigl( L^+_{\Div_X^d} G \bigr) (\ol{s}) \right]
\end{equation} 
(\cf \cite[Proposition VI.1.7]{FaScGeomLLC}), 
where the right hand side of \eqref{eq:Hckdc} is the groupoid whose objects are $\bigl( L_{\Div_X^d} G \bigr) (\ol{s})$ and morphisms from $g$ to $g'$ for $g,g' \in \bigl( L_{\Div_X^d} G \bigr) (\ol{s})$ are $(g_1,g_2 )\in \bigl( L^+_{\Div_X^d} G \bigr) (\ol{s})^2$ such that $g_1 g g_2^{-1}=g'$. 
For a geometric point $\ol{s}$ of $\Div_X^d$, we define 
\begin{equation}\label{eq:sidi}
 ( \ol{s}_i )_{1 \leq i \leq r}, \quad 
 ( d_i )_{1 \leq i \leq r} 
\end{equation}
by letting $\ol{s}_1, \ldots, \ol{s}_r$ be the different points in $d$ geometric point of $\Div_X^1$ given by $\ol{s}$, and $d_i$ be the multiplicity of $\ol{s}_i$ for $1 \leq i \leq r$. 
Then we have 
\[
 B_{\Div_X^d}^+(\ol{s}) \cong \prod_{i=1}^r B_{\Div_X^1}^+(\ol{s}_i), \quad 
 B_{\Div_X^d}(\ol{s}) \cong \prod_{i=1}^r B_{\Div_X^1}(\ol{s}_i) , 
\]
and $B_{\Div_X^1}^+(\ol{s}_i)$ is a complete discrete valuation ring, and $B_{\Div_X^1}(\ol{s}_i)$ is its quotient field for $1 \leq i \leq r$. 
Hence we obtain a bijection 
\begin{equation}\label{eq:HckXbij}
 \cHck_{G,\Div_X^d} (\ol{s})/{\sim} \cong (\check{\bX}^+)^r 
\end{equation}
by \eqref{eq:Hckdc} and the Cartan decomposition. 
For $\mu,\mu' \in \check{\bX}$, if $\mu -\mu' \in \bZ_{\geq 0} \cdot \check{\Delta}$, we say that $\mu'$ is less than or equal to $\mu$, and write $\mu' \leq \mu$. 
If $\mu' \leq \mu$ and $\mu' \neq \mu$, we write $\mu' < \mu$. 

\begin{dfn}\label{dfn:Hckmu}
Let $J$ be a set of $d$ elements, and $\mu_{\bullet}=(\mu_j)_{j \in J}$ with $\mu_j \in \check{\bX}^+$. 
Then, for $S \in \Perf_k$, we define $\cHck_{G,\Div_X^d,\leq \mu_{\bullet}}(S) \subset \cHck_{G,\Div_X^d}(S)$ by the following condition: 

For any geometric point $\ol{s}$ of $S$, if we define $( \ol{s}_i )_{1 \leq i \leq r}$ and $( d_i )_{1 \leq i \leq r}$
from $\ol{s} \to S \to \cHck_{G,\Div_X^d} \to \Div_X^d$ as \eqref{eq:sidi}, there is a decomposition $\coprod_{i=1}^r J_i$ of $J$ such that $\lvert J_i \rvert =d_i$ and the element of 
$(\check{\bX}^+)^r$ corresponding to $\ol{s} \to S \to \cHck_{G,\Div_X^d}$ under \eqref{eq:HckXbij} is less than or equal to $(\sum_{j \in J_i} \mu_j)_{1 \leq i \leq r}$ componentwisely. 
\end{dfn}

As \cite[Proposition VI.2.7]{FaScGeomLLC}, 
$\cHck_{G,\Div_X^d,\leq \mu_{\bullet}} \subset \cHck_{G,\Div_X^d}$ 
defined in Definition \ref{dfn:Hckmu} is a closed substack. 
For a small v-stack $S$ over $\Div_X^d$, we put 
\[
 \cHck_{G,S/\Div_X^d,\leq \mu_{\bullet}} =\cHck_{G,\Div_X^d,\leq \mu_{\bullet}} \times_{\Div_X^d} S . 
\]
For $\mu \in \check{\bX}^+$ and a small v-stack $S$ over $\Div_X^1$, we put 
\[
 \cHck_{G,S/\Div_X^1,\mu}= 
  \cHck_{G,S/\Div_X^1,\leq \mu} \setminus 
  \bigcup_{\mu' < \mu} \cHck_{G,S/\Div_X^1,\leq \mu'}. 
\]
For a geometric point $\ol{s}$ of $\Div_X^d$, 
we define $( \ol{s}_i )_{1 \leq i \leq r}$ and $( d_i )_{1 \leq i \leq r}$ as \eqref{eq:sidi}. 
Then, for 
$( \mu_i )_{1 \leq i \leq r} \in (\check{\bX}^+)^r$, 
the locally closed substack of $\cHck_{G,\ol{s}/\Div_X^d}$ determined by 
\[
 \prod_{i=1}^r \cHck_{G,\ol{s}_i/\Div_X^1,\mu_i} \subset 
 \prod_{i=1}^r \cHck_{G,\ol{s}_i/\Div_X^1} \cong 
 \cHck_{G,\ol{s}/\Div_X^d}
\]
is called the Schubert cell corresponding to $( \mu_i )_{1 \leq i \leq r}$. 

Let $G$ again be a general connected reductive algebraic group over $F$. 
Let $F'$ be a finite subextension of $F^{\mathrm{sep}}/F$ such that $G_{F'}$ is split\footnote{The t-structure of Definition \ref{dfn:Hckbdd} or Proposition \ref{prop:tstr} does not depend on the choice of $F'$.} and let $\Div_{X'}^d$ be the v-shaef of closed Cartier divisors of degree $d$ determined from $F'$. 
Let $S$ be a small v-stack over $\Div_X^d$. We put $S'=S \times_{\Div_X^d} \Div_{X'}^d$. 

\begin{dfn}\label{dfn:Hckbdd}
If the pullback of $A \in D_{\mathrm{et}} (\cHck_{G,S/\Div_X^d},\Lambda )$ to $D_{\mathrm{et}} (\cHck_{G,S'/\Div_X^d},\Lambda ) \cong D_{\mathrm{et}} (\cHck_{G_{F'},S'/\Div_{X'}^d},\Lambda )$ is obtained as the pushforward from the union of $\cHck_{G_{F'},S'/\Div_{X'}^d,\leq \mu_{\bullet}}$ for a finite number of $\mu_{\bullet} \in \check{\bX}^+$, we say that $A$ is bounded. 
\end{dfn}

Let $D_{\mathrm{et}} (\cHck_{G,S/\Div_X^d},\Lambda )^{\mathrm{bd}}$ be the full subcategory of bounded objects of $D_{\mathrm{et}} (\cHck_{G,S/\Div_X^d},\Lambda )$. 

\begin{dfn}\label{dfn:HckULA}
If $A \in D_{\mathrm{et}} (\cHck_{G,S/\Div_X^d},\Lambda )$ is bounded, and its pullback to $\Gr_{G,S/\Div_X^d}$ is universally locally acyclic over $S$, we say that $A$ is universally locally acyclic. 
\end{dfn}

For $A \in D_{\mathrm{et}} (\cHck_{G,S/\Div_X^d},\Lambda )$, 
the condition that $A$ is universally locally acyclic is equivalent to that $\mathrm{sw}^* A$ is universally locally acyclic (\cf \cite[Proposition 6.6.2]{FaScGeomLLC}). 

Let $D_{\mathrm{et}}^{\mathrm{ULA}} (\cHck_{G,S/\Div_X^d},\Lambda )$ denote the full subcategory of universally locally acyclic objects of $D_{\mathrm{et}} (\cHck_{G,S/\Div_X^d},\Lambda )$. 
We put $\rho=\frac{1}{2} \sum_{\alpha \in \Phi^+} \alpha$, where $\Phi^+$ denotes the set of positive roots of $\Phi$. 

\begin{prop}[{\cite[Proposition VI.7.1]{FaScGeomLLC}}]\label{prop:tstr}
There exists a unique t-structure $({}^p D^{\leq 0},{}^p D^{\geq 0})$ of $D_{\mathrm{et}} (\cHck_{G,S/\Div_X^d},\Lambda )^{\mathrm{bd}}$ such that $A \in D_{\mathrm{et}} (\cHck_{G,S/\Div_X^d},\Lambda )^{\mathrm{bd}}$ belongs to ${}^p D^{\leq 0}$ if and only if the following condition is satisfied: 

For any geometric point $\ol{s}$ of $S'$, 
if we define a sequence $( \ol{s}_i )_{1 \leq i \leq r}$ of different geometric points and $( d_i )_{1 \leq i \leq r}$ from $\ol{s} \to S' \to \Div_{X'}^d$ as \eqref{eq:sidi}, 
the cohomological dimension of the pullback of $A$ to the Schubert cell of $\cHck_{G,\ol{s}/\Div_X^d} \cong \cHck_{G_{F'},\ol{s}/\Div_{X'}^d}$ corresponding to arbitrary $( \mu_i )_{1 \leq i \leq r} \in (\check{\bX}^+)^r$ is less than or equal to $\sum_{i=1}^r \langle 2\rho, \mu_i \rangle$. 
\end{prop}

We define a t-structure $({}^p D^{\leq 0},{}^p D^{\geq 0})$
of $D_{\mathrm{et}} (\cHck_{G,S/\Div_X^d},\Lambda )^{\mathrm{bd}}$ as Proposition \ref{prop:tstr}, and put 
\begin{equation*}
\mathrm{Perv}(\cHck_{G,S/\Div_X^d},\Lambda )=
{}^p D^{\leq 0} \cap {}^p D^{\geq 0} \subset 
D_{\mathrm{et}} (\cHck_{G,S/\Div_X^d},\Lambda )^{\mathrm{bd}}.  
\end{equation*}
An object of $\mathrm{Perv}(\cHck_{G,S/\Div_X^d},\Lambda )$ is called a perverse sheaf. 

\begin{dfn}
We say that 
$A \in D_{\mathrm{et}} (\cHck_{G,S/\Div_X^d},\Lambda )$ 
is flat perverse if $A \otimes_{\Lambda}^{\bL} M$ is a perverse sheaf for any $\Lambda$-module $M$. 
\end{dfn}

\begin{dfn}
We write $\mathrm{Sat}(\cHck_{G,S/\Div_X^d},\Lambda )$ for the full subcategory of universally locally acyclic, flat perverse objects of $D_{\mathrm{et}} (\cHck_{G,S/\Div_X^d},\Lambda )$, and call it the Satake category. 
\end{dfn}

For a set $I$ of $d$ elements, we put 
\[
 \cHck_G^I=\cHck_{G,(\Div_X^1)^I/\Div_X^d}, \quad 
 \Gr_G^I=\Gr_{G,(\Div_X^1)^I/\Div_X^d}, \quad 
 \mathrm{Sat}_G^I(\Lambda)=\mathrm{Sat}(\cHck_G^I,\Lambda) .
\]
Let $\coprod_{i=1}^m I_i$ be a decomposition of a set $I$ of $d$ elements. In the following, we construct 
\begin{equation}\label{eq:satprod}
 \mathrm{Sat}_G^{I_1}(\Lambda) \times 
 \cdots \times \mathrm{Sat}_G^{I_m}(\Lambda) \to 
 \mathrm{Sat}_G^I(\Lambda) . 
\end{equation}
Let 
$(\Div_X^1)^{I;I_1,\ldots,I_m} \subset (\Div_X^1)^I$ 
be the closed sub-v-shaef defined by the condition that the $i$-th component and the $i'$-th component are different if 
$1 \leq j < j' \leq m$, $i \in I_j$ and $i' \in I_{j'}$. We put 
\[
\mathrm{Sat}_G^{I;I_1,\ldots,I_m}(\Lambda) =\mathrm{Sat}\left( \cHck_{G,(\Div_X^1)^{I;I_1,\ldots,I_m}/\Div_X^d},\Lambda \right) . 
\]
Since we have an isomorphism 
\[
 \cHck_G^I \times_{(\Div_X^1)^I} (\Div_X^1)^{I;I_1,\ldots,I_m} 
 \cong \prod_{j=1}^m \cHck_G^{I_j} \times_{\prod_{j=1}^m (\Div_X^1)^{I_j}} (\Div_X^1)^{I;I_1,\ldots,I_m} 
\]
over $(\Div_X^1)^{I;I_1,\ldots,I_m}$, we obtain 
\begin{equation}\label{eq:satprodop}
 \mathrm{Sat}_G^{I_1}(\Lambda) \times 
 \cdots \times \mathrm{Sat}_G^{I_m}(\Lambda) \to 
 \mathrm{Sat}_G^{I;I_1,\ldots,I_m}(\Lambda) 
\end{equation}
by exterior tensor products. 
By \cite[Proposition VI.9.3]{FaScGeomLLC}, the restriction functor $\mathrm{Sat}_G^I(\Lambda) \to \mathrm{Sat}_G^{I;I_1,\ldots,I_m}(\Lambda)$ is fully faithful. Furthermore, \eqref{eq:satprodop} factors through  $\mathrm{Sat}_G^I(\Lambda)$ and gives 
\eqref{eq:satprod} (\cf \cite[Proposition VI.9.4]{FaScGeomLLC}). In particular, we have 
\begin{equation}\label{eq:prodII}
\mathrm{Sat}_G^I(\Lambda) \times \mathrm{Sat}_G^I(\Lambda) \to \mathrm{Sat}_G^{I \sqcup I}(\Lambda). 
\end{equation}
The morphism 
\[
\cHck_G^I \cong \cHck_G^{I \sqcup I} \times_{(\Div_X)^{I \sqcup I}} (\Div_X)^I \to \cHck_G^{I \sqcup I}
\]
obtained from the natural surjection 
$I \sqcup I \to I$ gives $\mathrm{Sat}_G^{I \sqcup I}(\Lambda) \to \mathrm{Sat}_G^I(\Lambda)$. Composing it with \eqref{eq:prodII}, we obtain the fusion product 
\[
 * \colon \mathrm{Sat}_G^I(\Lambda) \times \mathrm{Sat}_G^I(\Lambda) \to \mathrm{Sat}_G^I(\Lambda) . 
\]
Let $A_1,A_2 \in \mathrm{Sat}_G^I(\Lambda)$. 
By the isomorphism swapping the $1$-st and $2$-nd components of $(\Div_X^1)^{I \sqcup I} =(\Div_X^1)^I \times (\Div_X^1)^I$, we obtain $A_1 * A_2 \cong A_2 * A_1$. 
We have a decomposition $\cHck_G^I=(\cHck_G^I)^{\mathrm{even}} \amalg (\cHck_G^I)^{\mathrm{odd}}$ into open and closed substacks 
according to the parity of the dimension of the Schubert cell contained in the base change to a geometric point of $(\Div_X^1)^I$. 
Using this decomposition, we decompose as $A_1=A_1^{\mathrm{even}}\oplus A_1^{\mathrm{odd}}$,  $A_2=A_2^{\mathrm{even}}\oplus A_2^{\mathrm{odd}}$, and define  
$c_{A_1,A_2} \colon A_1 * A_2 \cong A_2 * A_1$  
modifying $A_1 * A_2 \cong A_2 * A_1$ by the $-1$ multiplication on the decomposition factor $A_1^{\mathrm{odd}} * A_2^{\mathrm{odd}} \cong A_2^{\mathrm{odd}} * A_1^{\mathrm{odd}}$. 
By this, $\mathrm{Sat}_G^I(\Lambda)$ has a symmetric monoidal structure. 
For a group scheme $\mathcal{G}$ over a commutative ring $R$, let $\Rep_{R} (\mathcal{G})$ denote the category of finite projective $R$-modules with action of $\mathcal{G}$. 

\begin{thm}[{\cite[Theorem I.6.3]{FaScGeomLLC}}]\label{thm:geomS}
We fix $c \in \Lambda$ such that $c^2=q$. 
Then there is an equivalence 
\[
 \cS \colon \Rep_{\Lambda} ({}^L G_{\Lambda}^I) \stackrel{\sim}{\lra} 
 \mathrm{Sat}_G^I(\Lambda) 
\]
of symmetric monoidal categories, called the geometric Satake equivalence. 
\end{thm}

In Theorem \ref{thm:geomS}, the coefficient ring $\Lambda$ is torsion, but by taking the projective limit and extending the coefficients, we can obtain the equivalence in  $\ol{\bQ}_{\ell}$-coefficients. 
We explain a relation between this geometric Satake equivalence and the Satake isomorphism explained in Section \ref{sec:Satiso}. 
In the rest of this section, we assume that $G$ is unramified and consider the case where $I=\{*\}$ and $\Lambda=\ol{\bQ}_{\ell}$. 
In this case, there is another form of the geometric Satake equivalence 
\[
 \cS^{\mathrm{Witt}} \colon \Rep_{\ol{\bQ}_{\ell}} ({}^L G_{\Lambda}) \stackrel{\sim}{\lra} 
 \mathrm{Sat}_G^{\mathrm{Witt}}(\ol{\bQ}_{\ell}) 
 \]
proved by Zhu in \cite{ZhuAffGmix} (\cf \cite[\S 3.4]{XiZhCycSh}), where 
$\mathrm{Sat}_G^{\mathrm{Witt}}(\ol{\bQ}_{\ell})$ on the right hand side is a Satake category defined using an affine Witt vector Grassmann variety $\Gr_G^{\mathrm{Witt}}$, which is a perfect scheme over $k_{F}$, 
instead of the Beilinson--Drinfeld Grassmann variety. 
Roughly speaking, Fargues--Scholze's geometric Satake equivalence uses the part such that $\pi_F \neq 0$ among the whole untilts over $\cO_F$, and Zhu's geometric Satake equivalence uses the part such that $\pi_F = 0$\footnote{As explained below, Zhu's geometric Satake equivalence is directly related to the Satake isomorphism. On the other hand, the main advantage of Fargues--Scholze's geometric satake equivalence is that we can move untilts such that $\pi_F \neq 0$, which allows us to consider fusion products.}. 
We can construct an equivalence $\mathrm{Sat}_G^{\{*\}} \cong \mathrm{Sat}_G^{\mathrm{Witt}}$ of symmetric monoidal categories geometrically, so that the two geometric Satake equivalences are compatible (\cite{BanTwomon}). 
Furthermore, Zhu's geometric Satake equivalence and the Satake isomorphism is related by the commutative diagram 
\[
\xymatrix{
	K_0 (\Rep_{\ol{\bQ}_{\ell}} ({}^L G_{\Lambda})) 
	\ar[rr]^-{\sim}_-{K_0(\cS^{\mathrm{Witt}})} \ar@{->>}[d] & &
	K_0 ( \mathrm{Sat}_G^{\mathrm{Witt}}(\ol{\bQ}_{\ell}) ) \ar@{->>}[d] \\ 
	\Gamma ((\wh{G}_{\ol{\bQ}_{\ell}} \rtimes \sigma ) \sslash \wh{G}_{\ol{\bQ}_{\ell}},\cO)
	\ar[rr]^-{\sim}_-{\eqref{eq:HGammaiso}} & & \cH (G(F),K) 
}
\]
(\cite[\S 3.5]{XiZhCycSh}, \cite[Appendix]{ZhuSatrami}), 
where $K_0$ denotes the K-group of degree $0$, 
the left vertical morphism is a morphism given by characters, and the right vertical morphism is given by taking the trace of $q$-th power Frobenius element and the natural bijection $\Gr_G^{\mathrm{Witt}}(k_F) \cong G(F)/K$. 

\section{Derevied category of $\ell$-adic sheaves}

For geometrization of the local Langlands correspondence, we need $\ol{\bQ}_{\ell}$-sheaves on $\Bun_G$ that are related with smooth $\ol{\bQ}_{\ell}$-representations of $G_b(F)$ through the morphism \eqref{eq:Buntoquot}. 
The theory in \cite{SchEtdia} is about the sheaves of torsion coefficients, and although it is possible to extend the coefficients to $\ol{\bQ}_{\ell}$ by taking inverse limits, the sheaves obtained in such a way are related to the Banach representations of $G_b(F)$, which are different from what we want. 
There is also the problem that a smooth $\ol{\bQ}_{\ell}$-representation of $G_b(F)$ is not necessarily defined on $\ol{\bZ}_{\ell}$ in general. 
To overcome these problems, \cite[VII]{FaScGeomLLC} uses the idea of solid modules introduced by Clausen--Scholze to construct a derived category of sheaves in need. 
First, we define a solid module according to \cite[Definition 5.1]{SchLecCond} as follows:

\begin{dfn}
Let $M$ be a condensed module. 
we say that $M$ is solid if, for any inverse limit $S= \varprojlim_{i \in I} S_i$ of finite sets, the the morphism 
\[
 \Hom \left( \varprojlim_{i \in I} \bZ [S_i]_{\mathrm{c}}, M\right) \to M(S) 
\] 
induced from the natural morphism $S_{\mathrm{c}} \to \varprojlim_{i \in I} \bZ [S_i]_{\mathrm{c}}$ is an isomorphism,
where $\bZ [S_i]$ denotes the free module generated by $S_i$. 
\end{dfn}

With this definition in mind, we define a solid sheaf as follows: 

\begin{dfn}
Let $X$ be a spatial diamond and $\cF$ be a quasi-pro-etale sheaf of $\wh{\bZ}$-modules over $X$. 
We say that $\cF$ is solid if, for any quasi-pro-etale $j \colon U \to X$ which can be written as the projective limit of a directed inverse system $\{ j_i \colon U_i \to X \}_{i \in I}$ of quasi-compact, quasi-separated etale morphisms, 
the morphism  
\[
\Hom \left( \varprojlim_{i \in I} j_{i,!} \wh{\bZ}, \cF \right) \to \cF (U) 
\]
induced from the natural morphism $\wh{\bZ} \to j^* \varprojlim_{i \in I} j_{i,!} \wh{\bZ}$ is an isomorphism, 
where $j_{i,!} \wh{\bZ}$ denotes $\varprojlim_{n} j_{i,!} (\bZ/n\bZ)$. 
\end{dfn}

\begin{dfn}
Let $X$ be a small v-stack and $\cF$ be a v-sheaf of modules on $X$. 
We say that $\cF$ is solid if $f^* \cF$ is isomorphic to the pullback of a solid sheaf of $\wh{\bZ}$-modules on $X'_{\mathrm{qproet}}$ for any spatial diamond $X'$ and $f \colon X'\to X$. 
\end{dfn}

Let $\Lambda$ be a commutative ring over $\bZ_{\ell}$. 

\begin{dfn}
Let $X$ be a small v-stack, and let $D (X_{\mathrm{v}},\Lambda_{\mathrm{c},\ell})$ denote the derived category of sheaves of condensed $\Lambda_{\mathrm{c},\ell}$-module on $X_{\mathrm{v}}$. 
We define $D_{\solid} (X,\Lambda)$ as the full subcategory of 
$D (X_{\mathrm{v}},\Lambda_{\mathrm{c},\ell})$ consisting of every object $A$ of 
$D (X_{\mathrm{v}},\Lambda_{\mathrm{c},\ell})$ such that the cohomology sheaves in all degree of the images of $A$ in $D (X_{\mathrm{v}},\wh{\bZ})$ are solid. 
\end{dfn}

If $f \colon X \to Y$ is a morphism of small v-stacks, there are $4$ functors $f^*$, $f_*$, $\stackrel{\solid}{\otimes}$, $R\!\sHom_{\solid}$ on $D_{\solid} (-,\Lambda)$ as \cite[VII.2]{FaScGeomLLC}\footnote{Among these, the pull-back and push-out functors are denoted by the same symbols as those on v-site because they are consistent by \cite[Proposition VII.2.1]{FaScGeomLLC}.}, and additionally the left adjoint  functor 
\[
 f_{\natural} \colon D_{\solid} (X,\Lambda) \to D_{\solid} (Y,\Lambda) 
\]
of $f^*$ (\cf \cite[Proposition VII.3.1]{FaScGeomLLC}). 

\begin{dfn}
Let $X$ be an Artin v-stack. 
Let $D_{\lis} (X,\Lambda)$ be the smallest full triangulated subcategory of $D_{\solid} (X,\Lambda)$ that satisfies the following conditions\footnote{The subscript $\lis$ comes from the French word lisse.}: 
\begin{enumerate}
\item 
It is stable under arbitrary direct sum. 
\item 
It contains $f_{\natural} \Lambda_{\mathrm{c},\ell}$ 
for any separated, $\ell$-cohomologically smooth morphism $f \colon X' \to X$ that is representable by locally spatial  diamonds. 
\end{enumerate}
\end{dfn}

If $f \colon X \to Y$ is a morphism of Artin v-stacks, there are $4$ functors $f^*$, $f_{\lis,*}$, $\stackrel{\solid}{\otimes}$, $R\!\sHom_{\lis}$ on $D_{\lis} (-,\Lambda)$ as \cite[VII.6]{FaScGeomLLC}\footnote{The functors of pull-back and tensor are denoted by the same symbols as those on $D_{\solid} (-,\Lambda)$ because they are consistent by \cite[Proposition VII.6.2]{FaScGeomLLC}. The functors of push-out and homomorphism are obtained by composing with the lis-functor in \cite[Proposition VII.6.3]{FaScGeomLLC}.} 

For $b \in G(\breve{F})$, there is an equivalence of derived categories 
\begin{equation}\label{eq:lissmb}
D(G_b (F),\Lambda) \cong D_{\lis}(\Bun_G^{[b]},\Lambda) 
\end{equation}
induced by the morphism \eqref{eq:Buntoquot} (\cite[Proposition VII.7.1]{FaScGeomLLC}). 
Through this equivalence, we can consider geometrization of representations of $p$-adic reductive algebraic groups and operations between them. 
For example, the geometrization of a parabolic induction is given by the geometric Eisenstein functor (\cite[\S 9]{HamGeomES}, \cite[\S 4.3]{HaImDualcpx}\footnote{To be precise, coefficients are torsion in these literatures.}). 

As in \cite[IX.1]{FaScGeomLLC}, we can equip $D_{\solid} (-,\Lambda)$ and $D_{\lis} (-,\Lambda)$ with structures of condensed stable $\infty$-categories, and write them for $\cD_{\solid } (-,\Lambda)$ and $\cD_{\lis} (-,\Lambda)$. 
The above functors on $D_{\solid} (-,\Lambda)$ and $D_{\lis} (-,\Lambda)$ can be defined naturally on $\cD_{\solid} (-,\Lambda)$ and $\cD_{\lis} (-,\Lambda)$ as well. 

\section{Hecke  functor}

We put $\bZ_{\ell}[\sqrt{q}]=\bZ_{\ell}[X]/(X^2 -q)$. 
Using Theorem \ref{thm:geomS}, we consider 
\begin{equation}\label{eq:ZlS}
 \Rep_{\bZ_{\ell}[\sqrt{q}]} ({}^L G_{\bZ_{\ell}[\sqrt{q}]}^I) \to 
 \cD_{\solid} (\cHck_G^I,\bZ_{\ell}[\sqrt{q}] ); \ 
 V \mapsto \varprojlim_{n} \bD (\cS (V/\ell^n V))^{\vee}, 
\end{equation}
where $\bD$ is the relative Verdier dual with respect to 
$\cHck_G^I \to [(\Div_X^1)^I/L_{(\Div_X^1)^I}^+G]$. 
Let $\Lambda$ be a commutative ring over $\bZ_{\ell}[\sqrt{q}]$. 
\eqref{eq:ZlS} induces 
\begin{equation}\label{eq:LamSloc}
 \Rep_{\Lambda} ({}^L G_{\Lambda}^I) \to 
 \cD_{\solid} (\cHck_G^I,\Lambda )
\end{equation}
$\Lambda$-linearly. 
We define a small v-stack $\Hck_G^I$ by 
\begin{align*}
 \Hck_G^I (S)=\{ &\textrm{the groupoid of $G$-bundles $\cE_1$, $\cE_2$ over $X_S$,}\\ 
 &\textrm{$S \to (\Div_X^1)^I$ and an isomorphism $\cE_1|_{X_S \setminus D_S} \cong \cE_2|_{X_S \setminus D_S}$} \}
\end{align*}
for $S \in \Perf_k$. 
We call $\Hck_G^I$ the global Hecke stack. 
We define $p_1 \colon \Hck_G^I \to \Bun_G$ and $p_2 \colon \Hck_G^I \to \Bun_G \times (\Div_X^1)^I$ by sending 
$(\cE_1, \cE_2, S \to (\Div_X^1)^I, \cE_1|_{X_S \setminus D_S} \cong \cE_2|_{X_S \setminus D_S})$ to 
$\cE_1$ and 
$(\cE_2,S \to (\Div_X^1)^I)$, respectively. 
For an affinoid perfectoid space $S$ over $\Div_{X}^d$ such that 
$D_S$ is an affinoid adic space, we define 
$\Hck_G^I (S) \to \cHck_G^I (S)$ by the restriction to $B_{\Div_X^d}^+(S)$. 
Thus we obtain 
\begin{equation}\label{eq:Hckglloc}
\Hck_G^I \to \cHck_G^I. 
\end{equation}
By \eqref{eq:LamSloc} and the pull-back under \eqref{eq:Hckglloc}, we define 
\begin{equation*}
 \Rep_{\Lambda} ({}^L G_{\Lambda}^I) \to 
 \cD_{\solid} (\Hck_G^I,\Lambda );\ V \mapsto \cS'_V . 
\end{equation*}
For $V \in \Rep_{\Lambda} ({}^L G_{\Lambda}^I)$, the functor 
\begin{equation}\label{eq:Hckactdef}
 \cD_{\lis} (\Bun_G,\Lambda) \to 
 \cD_{\solid} (\Bun_G \times (\Div_X^1)^I,\Lambda) ;\ 
 A \mapsto p_{2,\natural} (p_1^* A \stackrel{\solid}{\otimes} \cS'_V) 
\end{equation}
gives 
\[
 T_V \colon \cD_{\lis} (\Bun_G,\Lambda) \to \cD_{\lis} (\Bun_G,\Lambda) 
\]
(\cf \cite[Proposition IX.2.1]{FaScGeomLLC}). 

\begin{thm}[{\cite[Theorem IX.2.2]{FaScGeomLLC}}]\label{thm:TVprev}
Let $V \in \Rep_{\Lambda} ({}^L G_{\Lambda}^I)$. 
Then the functor 
$T_V \in \End (\cD_{\lis} (\Bun_G,\Lambda))$ preserves 
limits, colimits and compact objects, and we have 
\[
 R\!\sHom_{\lis} (T_V(A),\Lambda) \cong T_{\mathrm{sw}*V^{\vee}} R\!\sHom_{\lis} (A,\Lambda)
\]
for $A \in \cD_{\lis} (\Bun_G,\Lambda)$, 
where $V^{\vee}$ denotes the dual of $V$, and $\mathrm{sw}$ denotes the involution of ${}^L G_{\Lambda}$ induced by \eqref{eq:swHck} and Theorem \ref{thm:geomS}. 
\end{thm}

The global Hecke stack used in the construction of the Hecke functor is closely related to the local Shimura variety and its generalization in a certain sense, moduli space of the mixed characteristic shtukas. Therefore, results for the Hecke functor contain information on the cohomology of these spaces. 
For example, using Theorem \ref{thm:TVprev}, we can deduce results on the finiteness of the $\ell$-adic etale cohomology of those spaces (\cf \cite[IX.3]{FaScGeomLLC}, \cite[\S 3]{ImaConv}). 

Let $\cD_{\lis} (\Bun_G,\Lambda)^{\omega}$ be the full condensed stable $\infty$-subcategory of compact objects of $\cD_{\lis} (\Bun_G,\Lambda)$. 
This gives a structure of a condensed  $\infty$-category on $\End (\cD_{\lis} (\Bun_G,\Lambda)^{\omega})$, and $\Aut (\cF)$ has a structure of a condensed animated group\footnote{Animation is an operation that creates an $\infty$-category from a category that satisfies appropriate conditions (\cf \cite[5.1.4]{CeScPurfl}). See \cite[Example 5.1.6]{CeScPurfl} for animated group.} for $\cF \in \End (\cD_{\lis} (\Bun_G,\Lambda)^{\omega})$. 
Let $\End (\cD_{\lis} (\Bun_G,\Lambda))^{\omega,BW_F^I}$ be the $\infty$-category of pairs of $\cF \in \End (\cD_{\lis} (\Bun_G,\Lambda)^{\omega})$ and a morphism $(W_F^I)_c \to \Aut (\cF)$ of condensed animated groups. 
Then, by \eqref{eq:Hckactdef}, we have a $\Rep_{\Lambda} ((\cW_F^I)_{\Lambda})$-linear monoidal  functor 
\begin{equation}\label{eq:Hecfun}
 T \colon \Rep_{\Lambda} ({}^L G_{\Lambda}^I) \to \End_{\Lambda}(\cD_{\lis} (\Bun_G,\Lambda)^{\omega})^{BW_F^I};\ V \mapsto T_V 
\end{equation}
which is functorial with respect to $I$ (\cf \cite[Corollary IX.2.4]{FaScGeomLLC}), 
where $\Rep_{\Lambda} ((\cW_F^I)_{\Lambda})$-linearity means that there is a natural isomorphism 
\begin{equation}\label{eq:Tlin}
 T_{V \otimes W}(A) \cong T_V(A) \otimes_{\Lambda} W
\end{equation}
for $W \in \Rep_{\Lambda} ((\cW_F^I)_{\Lambda})$ and $A \in \cD_{\lis} (\Bun_G,\Lambda)^{\omega}$. 

\section{Construction of semisimple $L$ parameters} 

Let $P$ be an open subgroup of $P_{F^*}$ which is a normal subgroup of $W_F$. 
For $n \geq 0$, let $F_n$ denote the free group generated by $n$ elements. 
For $n \geq 0$ and homomorphism $F_n \to W_F^0/P$, we consider the induced morphism 
\begin{equation}\label{eq:FnW0}
\cO (Z^1 (F_n,\wh{G})) \to \cO (Z^1 (W_F^0/P,\wh{G})) 
\end{equation}
(\cf Lemma \ref{lem:Z1rep}), where the action of $F_n$ on $\wh{G}$ is given by $F_n \to W_F^0/P$. 
The morphism \eqref{eq:FnW0} is $\wh{G}$-equivariant with respect to the conjugate action of $\wh{G}$. 
The morphism 
\[
 \mathrm{colim}_{(n,F_n \to W_F^0/P)} \cO (Z^1 (F_n,\wh{G})) \to \cO (Z^1 (W_F^0/P,\wh{G})) 
\]
obtained by taking colimit of \eqref{eq:FnW0} with respect to transition morphisms induced by morphisms $F_m \to F_n$ over $W_F^0/P$ is an isomorphism of rings with $\wh{G}$-actions. 
We put 
\[
 \mathrm{Exc} (W_F^0/P,\wh{G})=\mathrm{colim}_{(n,F_n \to W_F^0/P)} 
 \cO (Z^1 (F_n,\wh{G}))^{\wh{G}}. 
\]
The morphism 
\[
 \mathrm{Exc} (W_F^0/P,\wh{G}) \to \cO (Z^1 (W_F^0/P,\wh{G}))^{\wh{G}} 
\]
induced by \eqref{eq:FnW0} is a universally homeomorphism, and becomes an isomorphism after inverting $\ell$ (\cf \cite[VIII.3.2]{FaScGeomLLC}). 
Further, we put 
\[
 \mathrm{Exc} (W_F,\wh{G})=\varprojlim_{P} \mathrm{Exc} (W_F^0/P,\wh{G}) 
\]
and call it the excursion algebra. 
\begin{equation}\label{eq:ExO}
 \mathrm{Exc} (W_F,\wh{G}) \to \cO (Z^1 (W_F,\wh{G}))^{\wh{G}} 
\end{equation}
is also a universally homeomorphism, and becomes an isomorphism after inverting $\ell$. 
We put 
\[
 \cZ^{\mathrm{geom}} (G,\Lambda )=\pi_0 \End (\id_{\cD_{\lis} (\Bun_G,\Lambda)}) 
\]
and call it the geometric Bernstein center of $G$. 
In the following, we construct a natural morphism 
\begin{equation}\label{eq:ExcZgeom}
 \mathrm{Exc} (W_F,\wh{G})_{\Lambda} \to \cZ^{\mathrm{geom}} (G,\Lambda ). 
\end{equation}
Let $\cC_P  
 \subset 
 \cD_{\lis} (\Bun_G,\Lambda)^{\omega}$ be the full $\infty$-subcategory of 
$A \in \cD_{\lis} (\Bun_G,\Lambda)^{\omega}$ such that 
the action of $P^I$ on $T_V (A)$ is trivial for any $V \in \Rep_{\Lambda} ({}^L G_{\Lambda}^I)$. 
Since, for any $A \in \cD_{\lis} (\Bun_G,\Lambda)^{\omega}$, there exists $P$ such that $A \in \cC_P$ by \cite[Proposition IX.5.1]{FaScGeomLLC}, constructing 
\eqref{eq:ExcZgeom} is reduced to constructing 
\begin{equation}\label{eq:ExcCP}
 \mathrm{Exc} (W_F^0/P,\wh{G}) \to \pi_0 \End (\id_{\cC_P}) . 
\end{equation}
Let $Q$ be the image of $W_F \to \Aut (\wh{G})$. 
In the construction of \eqref{eq:ExcCP}, an excursion data\footnote{For the origin of the term excursion, see the diagram \eqref{eq:excdiag}.} defined below are used. 

\begin{dfn}
A tuple $(I,V,\alpha,\beta,\gamma)$ of 
a finite set $I$, $V \in \Rep_{\bZ_{\ell}} ((\wh{G} \rtimes Q)^I)$, 
$\alpha \colon 1 \to V|_{\wh{G}}$, 
$\beta \colon V|_{\wh{G}} \to 1$ and 
$\gamma \in (W_F^0/P)^I$ is called an excursion data, 
where $V|_{\wh{G}}$ denotes the restriction of $V$ by the diagonal morphism $\wh{G} \to  (\wh{G} \rtimes Q)^I$. 
\end{dfn}

Let 
$\cD =(I,V,\alpha,\beta,\gamma)$ be an excursion data. Let 
$\wt{\alpha} \colon \Ind_{\wh{G}}^{\wh{G} \rtimes Q} 1_{\wh{G}} \to V|_{\wh{G} \rtimes Q}$ and 
$\wt{\beta} \colon V|_{\wh{G} \rtimes Q} \to \Ind_{\wh{G}}^{\wh{G} \rtimes Q} 1_{\wh{G}}$ 
be the morphisms induced by $\alpha$ and $\beta$ respectively. 
We define an endomorphism $S_{\cD}$ of  
$\id_{\cC_P}$ by 
\begin{equation}\label{eq:excdiag}
\begin{split}
 \xymatrix{
 \id_{\cC_P}=T_{1_{\wh{G} \rtimes Q}} \ar[r] & 
 T_{1_{\wh{G} \rtimes Q}} \otimes \Ind_{1}^{Q} 1 \ar@{=}[r]^-{\stackrel{\eqref{eq:Tlin}}{\sim}} &
 T_{\Ind_{\wh{G}}^{\wh{G} \rtimes Q} 1_{\wh{G}}} \ar[r]^-{T_{\wt{\alpha}}} & 
 T_{V|_{\wh{G} \rtimes Q}} \ar@{=}[r]^-{\sim} & 
 T_V \ar[d]^-{\gamma} \\ 
 \id_{\cC_P} = T_{1_{\wh{G} \rtimes Q}} & 
 T_{1_{\wh{G} \rtimes Q}} \otimes \Ind_{1}^{Q} 1 \ar[l] & 
 T_{\Ind_{\wh{G}}^{\wh{G} \rtimes Q} 1_{\wh{G}}} \ar@{=}[l]_-{\stackrel{\eqref{eq:Tlin}}{\sim}} & 
 T_{V|_{\wh{G} \rtimes Q}} \ar[l]_-{T_{\wt{\beta}}} & 
 T_V \ar@{=}[l]_-{\sim}
}
\end{split}
\end{equation}
where the first and last morphisms are induced from the natural morphisms $1_Q \to \Ind_{1}^Q 1$ and $\Ind_{1}^Q 1 \to 1_Q$ of $Q$-representations, 
and two isomorphims $T_{V|_{\wh{G} \rtimes Q}} \cong T_V$ are given by the functoriality of $T$ with respect to $I$. 

For $f \in \cO (\wh{G}\backslash (\wh{G} \rtimes Q)^I/ \wh{G})$, let
$V_f \subset \cO ((\wh{G} \rtimes Q)^I/ \wh{G})$ be the sub-$(\wh{G} \rtimes Q)^I$-representation generated by $f$. We have 
$\alpha_f \colon 1 \to V_f|_{\wh{G}}$ given by $f$, 
and $\beta_f \colon V_f|_{\wh{G}} \to 1$ given by substituting $1 \in (\wh{G} \rtimes Q)^I$. 
For $\gamma \in (W_F^0/P)^I$, we put 
$\cD_{\gamma,f}=(I,V_f,\alpha_f,\beta_f,\gamma)$ and define 
\[
 \Theta_{\gamma} \colon \cO (\wh{G}\backslash (\wh{G} \rtimes Q)^I/ \wh{G}) \to \pi_0 \End (\id_{\cC_P}) ;\ f \mapsto S_{\cD_{\gamma,f}}. 
\]
If $F_n \to (W_F^0/P)^{\{1,\ldots,n\}}$ is given, 
let 
$w_1,\ldots,w_n \in (W_F^0/P)^{\{1,\ldots,n\}}$ be the image of the natural generators of $F_n$, and we put $w=(1,w_1,\ldots,w_n) \in (W_F^0/P)^{\{0,\ldots,n\}}$. 
Then, by 
\begin{align*}
 &\Theta_w \colon \cO (\wh{G}\backslash (\wh{G} \rtimes Q)^{\{0,\ldots,n\}}/ \wh{G}) \to \pi_0 \End (\id_{\cC_P}), \\ 
 &\cO (\wh{G}\backslash (\wh{G} \rtimes Q)^{\{0,\ldots,n\}}/ \wh{G}) 
 \otimes_{\cO (Q^{\{0,\ldots,n\}})} \cO (\{ w \}) 
 \cong \cO (Z^1 (F_n,\wh{G}) \sslash \wh{G}) , 
\end{align*}
we have 
$\cO (Z^1 (F_n,\wh{G}))^{\wh{G}} 
 \to \pi_0 \End (\id_{\cC_P})$. 
Taking the colimit of this, we obtain \eqref{eq:ExcCP}. 

Next, we construct a semisimple L-parameter using \eqref{eq:ExcZgeom}. 
Let $\bfC$ be an algebraically closed field over $\bZ_{\ell}[\sqrt{q}]$ and $\pi$ be a smooth irreducible representation of $G(F)$ over $\bfC$. 
Considering $\cF_{\pi} \in D_{\lis}(\Bun_G^{[1]},\bfC)$ corresponding to $\pi$ under \eqref{eq:lissmb}, 
$\cZ^{\mathrm{geom}} (G,\bfC )$ acts on $i^{[1]}_{\natural} \cF_{\pi} \in D_{\lis}(\Bun_G,\bfC)$ by a scalar multiple since $\pi$ is irreducible. 
By this action and \eqref{eq:ExcZgeom}, we have $\mathrm{Exc} (W_F,\wh{G})_{\bfC} \twoheadrightarrow \bfC$. 
Furthermore, we obtain a $\bfC$-valued point of $Z^1(W_F,\wh{G}) \sslash \wh{G}$ by the universally homeomorphism \eqref{eq:ExO}. 
The $\bfC$-valued points of $Z^1(W_F,\wh{G}) \sslash \wh{G}$ correspond to the closed $\wh{G}_{\bfC}$-orbits of $Z^1(W_F,\wh{G})_{\bfC}$, and further to the $\wh{G}(\bfC)$-conjugacy classes of semisimple $\ell$-adic L-parameters in $\bfC$-coefficients (\cite[Proposition VIII.3.2]{FaScGeomLLC}). 
Thus we have constructed the semisimple L-parameter of $\pi$. 

It is an important question whether the semisimple L-parameters constructed in this way are compatible with the already known constructions of the local Langlands correspondence.
In the following, we discuss the case $\bfC=\ol{\bQ}_{\ell}$. 
The compatibility is proved in \cite[Theorem IX.7.4]{FaScGeomLLC} if $G$ is $\GL_n$. 
If $F$ is mixed characteristic, it is proved in \cite[Theorem 6.6.1]{HKWKotloc} in the case where $G$ is an inner form of $\GL_n$, the cases of $\GSp_4$ and $\GU_n$ are studied in \cite{HamCompGT} and \cite{BMHNCompuni}, respectively. 
For general $G$, if $F$ is mixed characteristic and the irreducible representation of $G(F)$ has a non-zero parahoric fixed vector, then it is shown in \cite{LiComppar}. 
In the equal characteristic case, it is proved in \cite{LiHLgcompfun} that the Fargues--Scholze's construction is compatible with the Lafforgues' construction \cite{LafChtred} in the case of a function field, and coincides with the construction \cite{GeLaChtloc} by Genestier--Lafforgue. 

Other applications of the morphism \eqref{eq:ExcZgeom} include results on the finiteness of Hecke algebras of $p$-adic reductive algebraic groups in \cite{DHKMFinHec}, \cite{DHKMLLfamban}. 
It is remarkable that although \eqref{eq:ExcZgeom} is constructed geometrically, the assertion of the results on the finiteness of Hecke algebras are purely algebraic. 

\section{Singular support of coherent sheaves} 

In the geometrization conjecture of the local Langlands correspondence, a condition on the singular supports of coherent sheaves is used, following the formulation of the geometric Langlands correspondence in \cite{ArGaSinggLc}. 
In this section, we explain singular support. 

Let $S$ be a regular affine scheme and $X$ be a disjoint union of flat locally complete intersection affine schemes over $S$. 
Let $\bL_{X/S}$ be the cotangent complex of $X$ over $S$ (\cite[II (1.2.7.1)]{IllCcotI}). 
By \cite[III Proposition 3.2.6]{IllCcotI}, $\bL_{X/S}$ is isomorphic to a complex of vector bundles which is zero at degrees outside $[-1,0]$, locally on $X$. 
We put 
\[
 \mathrm{Sing}_{X/S}=\Spec \Sym_{\cO_X}^{\bullet} H^1 (\bL_{X/S}^{\vee}). 
\]
$\mathrm{Sing}_{X/S}$ is a group scheme affine over $X$ that represents the functor sending $X$-scheme $X'$ to $H^{-1} (\bL_{X/S} \otimes_{\cO_X}^{\bL} \cO_{X'})$, and has a natural action of $\bG_{\mathrm{m}}$. 
By \cite[III Proposition 3.1.2, Corollaire 3.2.7]{IllCcotI}, for $x \in X$, the triviality of $\mathrm{Sing}_{X/S} \times_X x$ is equivalent to the smoothness of $X \to S$ at $x$. 

By the product $\cO_X \otimes_{\cO_S} \cO_X \to \cO_X$, we view 
$\cO_X$ as an $\cO_X \otimes_{\cO_S} \cO_X$-module, and put 
\[
 \mathit{HH}^{\bullet} (X/S)=\Ext^{\bullet}_{\cO_X \otimes_{\cO_S} \cO_X} (\cO_X,\cO_X) . 
\]
We have 
\begin{equation}\label{eq:HLHH}
 H^1 (\bL_{X/S}^{\vee}) \cong \mathit{HH}^2 (X/S) 
\end{equation}
by \cite[III Th\'eor\`eme 1.2.3]{IllCcotI} and \cite[Theorem X.3.1]{MacHomology}. 
We write $D_{\mathrm{coh}}$ for the derived category of coherent sheaves. 
Let $\cE \in D_{\mathrm{coh}}^{\mathrm{b}}(X)$. 
By 
\[
 R\Hom_{\cO_X \otimes_{\cO_S} \cO_X} (\cO_X,\cO_X) \to 
 R\Hom_{\cO_X} (\cO_X \otimes_{\cO_X}^{\bL} \cE,\cO_X \otimes_{\cO_X}^{\bL} \cE) \cong 
 R\Hom_{\cO_X} (\cE,\cE) 
\]
we have 
\begin{equation}\label{eq:HHExt}
 \mathit{HH}^{\bullet} (X/S) \to \Ext^{\bullet}_{\cO_X} (\cE,\cE) . 
\end{equation}
By \eqref{eq:HLHH} and \eqref{eq:HHExt}, 
$\Ext^{\bullet}_{\cO_X} (\cE,\cE)$ has a structure of 
$\Sym_{\cO_X}^{\bullet} H^1 (\bL_{X/S}^{\vee})$-algebra, and define a 
$\bG_{\mathrm{m}}$-equivariant quasi-coherent sheaf $\mu \End (\cE)$ over $\mathrm{Sing}_{X/S}$. 
The support of $\mu \End (\cE)$ is called the singular support of $\cE$, and denoted by $\mathrm{SingSupp}(\cE)$. 
By \cite[Theorem VIII.2.9]{FaScGeomLLC}, 
$\cE$ is a perfect complex if and only if 
$\mathrm{SingSupp}(\cE)$ is contained in the zero section of $\mathrm{Sing}_{X/S}$. 

\section{Categorical conjecture}

We assume one of the following: 
\begin{itemize}
\item 
$\Lambda$ is an algebraic field extension of $\bQ_{\ell}[\sqrt{q}]$. 
\item 
$\Lambda$ is the ring of integers in an algebraic field extension of $\bQ_{\ell}[\sqrt{q}]$, 
and $\ell$ does not divide the order of $\pi_1(\wh{G})_{\mathrm{tor}}$. 
\end{itemize}

Let $\Perf (\mathrm{LP}_{G,\Lambda})$ denote the stable $\infty$-category of perfect complexes over $\mathrm{LP}_{G,\Lambda}$. 
Then, by \cite[Theorem X.0.1]{FaScGeomLLC}, 
we can naturally construct a $\Lambda$-linear action of 
$\Perf (\mathrm{LP}_{G,\Lambda})$ on 
$\cD_{\lis} (\Bun_G,\Lambda)^{\omega}$ such that its composition with 
\begin{equation*}
 \Rep_{\Lambda} ({}^L G_{\Lambda}^I) \to \Perf (\mathrm{LP}_{G,\Lambda})^{BW_F^I} ;\ 
 V \mapsto (\cO_{Z^1 (W_F,\wh{G})_{\Lambda}} \otimes V )/\wh{G}_{\Lambda}
\end{equation*}
coincides with the functor \eqref{eq:Hecfun}. 
This action is called spectral action. 
This is what should exist if we assume the categorical conjecture explained later, and we can actually construct it, but we can also use the spectral action to define a functor that is expected to give the equivalence of categories in the categorical conjecture. 
The spectral action is very useful, and its application to the Eichler--Shimura relation for local Shimura varieties is given in \cite{KosESloc}, and its application to the vanishing of the cohomology of local Shimura varieties and Shimura varieties is given in \cite{KosOngenlg}. 

By \cite[Corollary VIII.2.3]{FaScGeomLLC}, there is a natural immersion 
\begin{equation}\label{eq:Singemb}
	\mathrm{Sing}_{Z^1 (W_F,\wh{G})_{\Lambda}/\Lambda} \hookrightarrow 
	\Lie (\wh{G})_{\Lambda}^* \times_{\Lambda} Z^1 (W_F,\wh{G})_{\Lambda}. 
\end{equation}
Let $\cN_{\wh{G}} \subset \Lie (\wh{G})^*$ be the closed subset given by the union of all $\wh{G}$-orbits whose closure contain the origin. 
Let $\cD_{\mathrm{coh},\mathrm{Nilp}}^{\mathrm{b},\mathrm{qc}}(\mathrm{LP}_{G,\Lambda})$ be the 
stable $\infty$-category of bounded complexes of quasi-coherent sheaves over 
$\mathrm{LP}_{G,\Lambda}$ whose cohomology in each degree satisfies the following conditions: 
\begin{itemize}
\item 
It has support in a finite number of connected components.
\item 
It is coherent on each connected component. 
\item 
After the pullback to $Z^1 (W_F,\wh{G})_{\Lambda}$, its singular support is contained in  $\cN_{\wh{G},\Lambda} \times_{\Lambda} Z^1 (W_F,\wh{G})_{\Lambda}$ under the immersion  \eqref{eq:Singemb}\footnote{This condition regarding singular support is automatically satisfied if $\Lambda$ is an algebraic field extension of $\bQ_{\ell}[\sqrt{q}]$ (\cite[Proposition VIII.2.11]{FaScGeomLLC}).}. 
\end{itemize}

Let $\Perf^{\mathrm{qc}} (\mathrm{LP}_{G,\Lambda})$ be the full stable $\infty$-subcategory of $\Perf (\mathrm{LP}_{G,\Lambda})$ consisting of every object which has support in a finite number of connected components. 
By the extension of the spectral action preserving colimits, 
we have an action of $\Ind \Perf^{\mathrm{qc}} (\mathrm{LP}_{G,\Lambda})$ on $\cD_{\lis} (\Bun_G,\Lambda)$. 
The next is the categorical conjecture formulated by Fargues--Scholze.

\begin{conj}\label{conj:geomcj}
We assume that $G$ is quasi-split, and fix a Whittaker datum $(B,\psi)$ in $\Lambda$-coefficients for $G$. 
By \eqref{eq:lissmb}, we view $\cInd_{R_{\mathrm{u}}(B)(F)}^{G(F)} \psi$ as an object of  
$\cD_{\lis}(\Bun_G^{[1]},\Lambda)$, and put 
\[
\cW_{\psi} =i^{[1]}_{\natural} (\cInd_{R_{\mathrm{u}}(B)(F)}^{G(F)} \psi ) \in \cD_{\lis}(\Bun_G,\Lambda) . 
\]
Then, the restriction to $\cD_{\lis} (\Bun_G,\Lambda)^{\omega}$ of the right adjoint functor of 
\[
 a_{\psi} \colon \Ind \Perf^{\mathrm{qc}} (\mathrm{LP}_{G,\Lambda}) \to 
 \cD_{\lis} (\Bun_G,\Lambda)
\]
given the action of $\Ind \Perf^{\mathrm{qc}} (\mathrm{LP}_{G,\Lambda})$ to $\cW_{\psi}$ 
induces an equivalence 
\[
 \cD_{\lis} (\Bun_G,\Lambda)^{\omega} \cong 
 \cD_{\mathrm{coh},\mathrm{Nilp}}^{\mathrm{b},\mathrm{qc}}(\mathrm{LP}_{G,\Lambda}) 
\]
of stable $\infty$-categories that is compatible with the actions of $\Perf (\mathrm{LP}_{G,\Lambda})$. 
\end{conj}

Since the moduli of L-parameters can be viewed as the moduli of local systems on the Fargues--Fontaine curve, in Conjecture \ref{conj:geomcj}, the left side is the category of etale objects on the moduli of coherent objects, and the right side is the category of coherent objects on the moduli of etale objects, so we can see a kind of symmetry. 
If $G$ is a torus, Conjecture \ref{conj:geomcj} is proved in \cite{ZoucatFartori}. 

We explain some of the relevant conjectures. 
In the following, $\widehat{G}$ and ${}^L G$ are considered over $\ol{\bQ}_{\ell}$ and denoted by the same symbols. 
We assume that $G$ is quasi-split. 
We fix a Whittaker data $(B,\psi)$ in $\ol{\bQ}_{\ell}$-coefficients for $G$, and write $\fw$ for its $G(F)$-conjugacy class.
First, we describe the Fargues conjecture stated in \cite[Conjecture 4.4]{FarGover}\footnote{In \cite[Conjecture 4.4]{FarGover}, a conjecture on the local-global compatibility is also stated, but we omit it here.}. 
Let $\phi$ be a discrete L-parameter of $\SL_2$-type in $\ol{\bQ}_{\ell}$-coefficients for $G$, and let $\varphi$ be the corresponding $\ell$-adic L-parameter. 
We put $S_{\varphi}=Z_{\wh{G}}(\varphi)$. 
Then $S_{\phi}=S_{\varphi}$ as in the proof of \cite[Proposition 3.15]{BMIYJMmor}\footnote{If $\phi$ is not discrete, then $S_{\phi}$ and $S_{\varphi}$ are different in general (\cite[Example 3.8]{BMIYJMmor}).}. 
For a cocharacter $\mu$ of $G_{F^{\mathrm{sep}}}$, let $E_{\mu}$ be the reflex field of $G(F^{\mathrm{sep}})$-conjugacy class of $\mu$, and we put $r_{\mu}=\Ind_{\wh{G} \rtimes W_{E_{\mu}}}^{{}^L G} r'_{\mu}$, where $r'_{\mu}$ is the $\wh{G} \rtimes W_{E_{\mu}}$-representation of highest weight $\mu$ (\cite[(2.1.2)]{KotShtw}). 

\begin{conj}[Fargues conjecture]\label{conj:Farconj}
There exists $\cF_{\varphi} \in \cD_{\lis} (\Bun_G,\ol{\bQ}_{\ell})$ with action of $S_{\varphi}$ which satisfy the following conditions: 
\begin{enumerate}
\item\label{en:Heceigen} 
For a cocharacter $\mu$ of $G_{F^{\mathrm{sep}}}$, we have 
\[
 p_{2,\natural} (p_1^* \cF_{\varphi} \stackrel{\solid}{\otimes} \cS'_{r_{\mu}}) \cong \cF_{\varphi} \boxtimes (r_{\mu} \circ \varphi ) \in \cD_{\solid} (\Bun_G \times \Div_X^1,\ol{\bQ}_{\ell}) . 
\]
\item\label{en:bdescLLC} 
For a basic element $b \in G(\breve{F})$, we have 
\[
 i^{[b],*} \cF_{\varphi} \cong \bigoplus_{\rho \in \Irr (S_{\phi}),\rho|_{Z(\wh{G})^{W_F}}=\kappa (b)} \rho \otimes \pi_{\varphi,b,\rho} , 
\]
where $\pi_{\varphi,b,\rho}$ denotes $\iota_{\fw}^{-1} (\rho) \in \Pi_{[b],\phi}$ in Conjecture \ref{conj:LL}. 
\item\label{en:Fcusp}
If $\varphi$ is cuspidal, $\cF_{\varphi} =j_{\natural} j^* \cF_{\varphi}$ where $j \colon \Bun_G^{\mathrm{ss}} \hookrightarrow \Bun_G$ is the natural open immersion. 
\end{enumerate}
\end{conj}

The condition \ref{en:Heceigen} of Conjecture \ref{conj:Farconj} is called the Hecke eigensheaf property. 
If the L-parameter is cuspidal, Conjecture \ref{conj:Farconj} is proved in \cite{GINsemi} for $\GL_2$ and minuscule $\mu$, and in \cite{AnLBAvGLn}, \cite {HansclocSh} for $GL_n$. 
In \cite{NgucatGLn}, it is proved for L-parameters of $\GL_n$ which are not necessarily cuspidal, under certain genericity condition. 

Next, we explain a relationship between Conjecture \ref{conj:Farconj} and $a_{\psi}$ appearing in Conjecture \ref{conj:geomcj}. 
Let 
\[
f_{\varphi} \colon \Spec \ol{\bQ}_{\ell} \to \mathrm{LP}_{G,\ol{\bQ}_{\ell}}  
\]
be the morphism given by $\varphi$, and we put 
\[
 \cE_{\varphi} =f_{\varphi,*} \ol{\bQ}_{\ell} \in 
 \cD_{\mathrm{qcoh}} (\mathrm{LP}_{G,\ol{\bQ}_{\ell}})
 \cong \Ind \Perf^{\mathrm{qc}} (\mathrm{LP}_{G,\ol{\bQ}_{\ell}}), 
\]
where $\cD_{\mathrm{qcoh}}$ denotes the stable $\infty$-category of quasi-coherent sheaves, and its equivalence with $\Ind \Perf^{\mathrm{qc}}$ follows from \cite[Corollary 3.22]{BFNIntDri}. 
Since $f_{\varphi}$ factors through $i_{\varphi} \colon [\Spec \ol{\bQ}_{\ell}/S_{\varphi}] \hookrightarrow \mathrm{LP}_{G,\ol{\bQ}_{\ell}}$, there is a natural action of $S_{\varphi}$ on $\cE_{\varphi}$. 
Further we put  
\[
 \mathrm{Aut}_{\varphi}=a_{\psi} (\cE_{\varphi}) \in \cD_{\lis} (\Bun_G,\ol{\bQ}_{\ell}). 
\]
Then the action of $S_{\varphi}$ on $\mathrm{Aut}_{\varphi}$ is also induced. 

\begin{conj}
$\mathrm{Aut}_{\varphi}$ gives $\cF_{\varphi}$ in Conjecture \ref{conj:Farconj}. 
\end{conj}

It is easy to verify that $\mathrm{Aut}_{\varphi}$ satisfies the Hecke eigensheaf property, and a question is whether it satisfies the condition of \ref{en:bdescLLC} in Conjecture \ref{conj:Farconj}. 

In the following, we assume that $\phi$ is cuspidal, and explain a conjecture that categorifies the bijection $\iota_{\fw}$ in Conjecture \ref{conj:LLb}. 
The unramified twists of $\varphi$ gives a connected component $C_{\varphi}$ of $\mathrm{LP}_{G,\ol{\bQ}_{\ell}}$. 
Considering the idempotent element of $\mathrm{Exc} (W_F,\wh{G})$ determined from the connected component corresponding to $C_{\varphi}$ under \eqref{eq:ExO}, let 
\[
\cD_{\lis}^{C_{\varphi}} (\Bun_G,\ol{\bQ}_{\ell})^{\omega} \subset \cD_{\lis} (\Bun_G,\ol{\bQ}_{\ell})^{\omega}
\]
be the direct summand given by the action of that idempotent. 
Let $\varphi^{\rmab} \colon W_F \to {}^L G^{\rmab}$ be the L-parameter of $G^{\rmab}$ determined from $\varphi$, 
and $\chi$ be the corresponding character of $Z(G)^{\circ}(F)$. 
Let $\cD_{\lis}^{C_{\varphi},\chi} (\Bun_G,\ol{\bQ}_{\ell})^{\omega}$ be the full stable $\infty$-subcategory of objects of $\cD_{\lis}^{C_{\varphi}} (\Bun_G,\ol{\bQ}_{\ell})^{\omega}$ such that, after the pullback to $\Bun_G^{[b]}$, the actions of $Z(G_b)^{\circ}(F) \cong Z(G)^{\circ}(F)$ on the corresponding representations are $\chi$ for any $[b] \in B(G)_{\mathrm{basic}}$. 
For $W \in \Perf ([\Spec \ol{\bQ}_{\ell}/S_{\varphi}])$, we write $\mathrm{Act}_W$ for the spectral action of the object of $\Perf (\mathrm{LP}_{G,\ol{\bQ}_{\ell}})$ given by the pullback of $W$ under the natural morphism $C_{\varphi} \to [\Spec \ol{\bQ}_{\ell}/S_{\varphi}]$ and the pushforward along $C_{\varphi} \hookrightarrow \mathrm{LP}_{G,\ol{\bQ}_{\ell}}$. 

\begin{conj}
There exists a unique generic representation $\pi$ of $G(F)$ whose L-parameter is $\varphi$, and the functor 
\[
 \Perf ([\Spec \ol{\bQ}_{\ell}/S_{\varphi}]) \to 
 \cD_{\lis}^{C_{\varphi},\chi} (\Bun_G,\ol{\bQ}_{\ell})^{\omega} ;\ 
 W \mapsto \mathrm{Act}_{W}(i^{[1]}_{\natural} \pi) 
\]
is an equivalence of stable $\infty$-categories. 
Furthermore, for $\rho \in \Irr (S_{\varphi})$, if we take $[b]\in B(G)_{\mathrm{basic}}$ such that the restriction of $\rho$ to $Z(\wh{G})^{W_F}$ is equal to $\kappa_G([b])$, then, the object of $\cD_{\lis}^{C_{\varphi},\chi} (\Bun_G,\ol{\bQ}_{\ell})^{\omega}$ given by $\rho \in \Irr (S_{\varphi})$ under the above equivalence coincides with the image of $\iota_{\fw}^{-1} (\rho) \in \Pi_{[b],\phi}$ in Conjecture \ref{conj:LL} under $i^{[b]}_{\natural}$. 
\end{conj}


\section*{Acknowledgements} 
The author is grateful to Jean-Fran\c{c}ois Dat, Laurent Fargues, Teruhisa Koshikawa, Peter Scholze and Xinwen Zhu for answering questions during the writing of this paper. 
He also thanks the referees for their helpful comments.


\vspace*{1em}

\noindent
Naoki Imai\\
Graduate School of Mathematical Sciences, The University of Tokyo, 
3-8-1 Komaba, Meguro-ku, Tokyo, 153-8914, Japan \\
naoki@ms.u-tokyo.ac.jp 

\end{document}